\documentclass[a4paper, 12pt]{{amsart}}
\usepackage{amsmath, mathtools}
\usepackage{amscd}
\usepackage{amssymb}
\usepackage{amsthm}
\usepackage{ascmac, color}
\usepackage[dvipdfmx]{graphicx, xcolor, pdfpages}
\usepackage[all]{xy}
\usepackage{fancybox}

\newtheorem{lemma}{{\sc Lemma}}[section]
\newtheorem{corollary}[lemma]{{\sc Corollary}}
\newtheorem{proposition}[lemma]{{\sc Proposition}}
\newtheorem{theorem}[lemma]{{\sc Theorem}}
\theoremstyle{definition}
\newtheorem{remark}[lemma]{{\sc Remark}}

\newtheorem{conjecture}[lemma]{{\sc Conjecture}}

\numberwithin{equation}{section}

\def\Ga{{\mathfrak{a}}}
\def\Gb{{\mathfrak{b}}}
\def\Gc{{\mathfrak{c}}}

\def\Gg{{\mathfrak{g}}}
\def\Gh{{\mathfrak{h}}}

\def\Gm{{\mathfrak{m}}}
\def\Gn{{\mathfrak{n}}}

\def\BA{{\mathbb{A}}}

\def\BC{{\mathbb{C}}}

\def\BF{{\mathbb{F}}}

\def\BQ{{\mathbb{Q}}}

\def\BZ{{\mathbb{Z}}}
\def\CA{{\mathcal A}}
\def\CB{{\mathcal B}}

\def\DD{{\mathcal D}}
\def\CE{{\mathcal E}}
\def\CF{{\mathcal F}}

\def\CO{{\mathcal O}}

\def\CK{{\mathcal K}}
\def\CL{{\mathcal L}}

\def\CR{{\mathcal R}}

\def\CV{{\mathcal V}}

\def\CX{{\mathcal X}}

\def\ad{{\mathop{\rm ad}\nolimits}}

\def\coh{{\mathop{\rm coh}\nolimits}}
\def\Cok{{\mathop{\rm Cok}\nolimits}}

\def\End{\mathop{\rm{End}}\nolimits}
\def\eq{{\mathop{\rm{eq}}\nolimits}}
\def\Ext{{\mathop{\rm Ext}\nolimits}}
\def\For{{\mathop{\rm For}\nolimits}}
\def\Fr{{\mathop{\rm Fr}\nolimits}}

\def\Har{{\mathop{\rm Har}\nolimits}}
\def\Hom{\mathop{\rm Hom}\nolimits}

\def\id{\mathop{\rm id}\nolimits}

\def\Ind{\mathop{\rm Ind}\nolimits}
\def\inte{{\mathop{\rm{int}}\nolimits}}
\def\Image{\mathop{\rm Im}\nolimits}

\def\Ker{\mathop{\rm Ker\hskip.5pt}\nolimits}

\def\Mod{\mathop{\rm Mod}\nolimits}

\def\rank{{\mathop{\rm rank}\nolimits}}
\def\res{\mathop{\rm res}\nolimits}

\def\Trace{{\mathop{\rm Trace}\nolimits}}
\def\Tor{{\rm{Tor}}}

\def\Bk{{\mathbf{k}}}

\begin{document}
\title[quantized flag manifolds]
{
Differential operators on quantized flag manifolds at roots of unity III}
\author{Toshiyuki TANISAKI}
\address{
9-3-12 Jiyugaoka, Munakata, Fukuoka, 811-4163 Japan}
\email{ttanisaki@icloud.com}
\begin{abstract}
We describe the cohomology of the sheaf of twisted differential operators on the quantized flag manifold 
at a root of unity whose order is a prime power.
It follows from this and our previous results that 
for the De Concini-Kac type quantized enveloping algebra, where the parameter $q$ is specialized to a root of unity whose order is a prime power, 
the number of irreducible modules with a certain specified central character coincides with the dimension of the total cohomology group of the corresponding Springer fiber. 
This gives a weak version of a conjecture of Lusztig concerning non-restricted representations of the quantized enveloping algebra.
\end{abstract}

\maketitle

\section{Introduction}
\subsection{}
Let $\Gg_k$ be the Lie algebra of a connected semisimple algebraic group over an algebraically closed field $k$ of positive characteristic.
In a celebrated work \cite{BMR} 
Bezrukavnikov, Mirkovi\'c and Rumynin established 
two important results concerning the sheaf $\DD$ of twisted differential operators on the corresponding flag manifold.
One is 
the Beilinson-Bernstein type derived equivalence between the category of certain representations of $\Gg_k$ and that of $\DD$-modules.
The other is the split Azumaya property 
of  $\DD$ over a certain central subalgebra.
They obtained a significant application of these results 
to the  non-restricted representation theory of $\Gg_k$.

The present work is the third part of  the series of papers giving an analogue of \cite{BMR} using quantized flag manifolds and quantized enveloping algebras at roots of unity
instead of ordinary flag manifolds and ordinary enveloping algebras in positive characteristics.
\subsection{}
Let $G$ be a connected simply-connected simple algebraic group over $\BC$.
Take a maximal torus $H$ of $G$ and a Borel subgroup $B^-$  of $G$ containing $H$.
Let $P$ and $Q$ be the weight lattice and the root lattice respectively.
Using the corresponding quantum group we can construct  a non-commutative projective scheme $\CB_q$ which is called the 
quantized flag manifold
(see \cite{M}, \cite{R}, \cite{LR}).
We are concerned with the situation where the parameter $q$ is specialized to a primitive $\ell$-th root of unity $\zeta\in\BC^\times$.
Here, $\ell>1$ is an odd integer which is prime to $|P/Q|$.
When $G$ is of type $G_2$, we further assume that $\ell$ is prime to 3.
For each $t\in H$ we have a sheaf $\DD_{\CB_\zeta,t}$ of twisted differential operators on $\CB_\zeta$.
The split Azumaya property of $\DD_{\CB_\zeta,t}$ was already established in \cite{T1}.
We also expect that 
the Beilinson-Bernstein type derived equivalence holds for 
$\DD_{\CB_\zeta,t}$.
Namely, we conjecture that the category of coherent $\DD_{\CB_\zeta,t}$-modules is derived equivalent to the category of finitely generated $U_{\zeta,t}$-modules when $t$ satisfies a certain regularity condition.
Here $U_\zeta$ denotes the De Concini-Kac type quantized enveloping algebra at $q=\zeta$, 
and we set
$U_{\zeta,t}=U_\zeta\otimes_{Z_{\Har}(U_\zeta)}\BC$, 
where 
$Z_{\Har}(U_\zeta)$ denotes the Harish-Chandra center of 
$U_\zeta$ and $Z_{\Har}(U_\zeta)\to\BC$ is the character associated to $t$.
As in \cite{BMR} this conjecture follows if we can show
\begin{equation}
\label{eq:RG}
R\Gamma(\CB_\zeta,\DD_{\CB_\zeta,t})
\cong
U_{\zeta,t}.
\end{equation}
In \cite{T2} we have shown that \eqref{eq:RG} is a consequence of 
\begin{equation}
\label{eq:RI}
R\Ind(\CO_\zeta(B^-)_\ad)\cong 
\CO_\zeta(G)_\ad\otimes_{\CO(H/W\bullet)}\CO(H).
\end{equation}
Here, $\CO_\zeta(G)$ and
$\CO_\zeta(B^-)$
denote the quantized coordinate algebras of $G$ and $B^-$ at $q=\zeta$ respectively.
We add the subscript ``$\ad$'' since they are regarded as comodules over themselves via the adjoint action.
We denote the coordinate algebras of $H$ and $H/W\bullet$ by 
 $\CO(H)$ and $\CO(H/W\bullet)$ respectively, 
 where $H/W\bullet$ is the quotient of $H$ with respect to a twisted action of the Weyl group $W$.
The induction functor from the category of $\CO_\zeta(B^-)$-comodules to that of
$\CO_\zeta(G)$-comodules is denoted as $\Ind$.

In this paper we restrict ourselves to the case when $\ell$ is a prime power, and derive  a weak form of \eqref{eq:RI}
from the corresponding fact for $q=1$.
In the course of the proof we use some facts from the theory of canonical bases (global crystal bases) in an essential way.
The standard resolution given in \cite{APW}
also plays a crucial role in our argument concerning $R\Ind$.
We note that for $G$ of type $A$ we can show \eqref{eq:RI} without assuming that $\ell$ is a prime power by a totally different method.
This will appear in another paper.

From the weak form of \eqref{eq:RI} mentioned above
we obtain the following results.

\begin{theorem}
\label{thm:Int1}
Assume that $\ell$ is a power of a prime.
Let  $t\in H$ and assume that the order of $t^\ell$ is finite and prime to $\ell$.
Then we have
\[
R\Gamma(\CB_\zeta,\DD_{\CB_\zeta,t})
\cong
U_{\zeta,t}.
\]
\end{theorem}

\begin{theorem}
\label{thm:Int2}
Let $\ell$ and  $t$ be as in Theorem \ref{thm:Int1}.
We further assume $|W\bullet t|=|W|$.
Then we have an equivalence
\[
D^b(\Mod_{\coh}(\DD_{\CB_\zeta,t}))
\cong
D^b(\Mod_f(U_{\zeta,t}))
\]
between the bounded derived categories.
Here, $\Mod_{\coh}(\DD_{\CB_\zeta,t})$
$($resp.\ 
$\Mod_f(U_{\zeta,t}))$
denotes the category of coherent $\DD_{\CB_\zeta,t}$-modules
$($resp.\ 
finitely generated $U_{\zeta,t}$-modules$)$.
\end{theorem}

\subsection{}
Let us describe the application of our result to the representation theory of $U_\zeta$.
Let $Z(U_\zeta)$ denote the center of $U_\zeta$.
For a character $\xi:Z(U_\zeta)\to\BC$ we set $U_\zeta(\xi)=
U_\zeta/U_\zeta\Ker(\xi)$.
Then any irreducible $U_\zeta$-module is a $U_\zeta(\xi)$-module for some $\xi$.
We can associate to $\xi$ a pair $([t],g)\in (H/W\bullet)\times G$ such that $t^\ell$ is conjugate to $g_{s}$ in $G$.
Here, the Jordan decomposition of $g\in G$ is denoted as $g=g_{s}g_{u}$.
Note that $[t]$ corresponds to the restriction of $\xi$ to the Harish-Chandra center, and $g$ is determined from the restriction of $\xi$ to the Frobenius center.
Let $\xi'$ be another central character whose associated pair is of the form  $([t],g')$.
It is known from the theory of quantum adjoint action (\cite{DKP}) 
that 
$U_\zeta(\xi)\cong U_\zeta(\xi')$ if $g$ and $g'$ are conjugate in $G$.
Let $\CB_0$ be the (ordinary) flag manifold for $G_0=Z_G(g_{s})$, and let 
$\CB_0^{g_u}$ be the Springer fiber associated the the unipotent element $g_u\in G_0$.
Lusztig's conjecture \cite{LK} implies that if $t$ is regular, then there exists a natural isomorphism
\begin{equation}
\label{eq:KK}
K(\Mod_f(U_\zeta(\xi)))\cong K(\Mod_\coh(\CO_{\CB_0^{g_u}}))
\end{equation}
between the Grothendieck groups, where
$\Mod_f(U_\zeta(\xi))$ denotes the category of finitely generated $U_\zeta(\xi)$-modules, and 
$\Mod_\coh(\CO_{\CB_0^{g_u}})$ denotes that of 
coherent $\CO_{\CB_0^{g_u}}$-modules.
This isomorphism implies that the number of irreducible modules in $\Mod_f(U_\zeta(\xi))$ coincides with the dimension of the total cohomology group of the Springer fiber $\CB_0^{g_u}$.
By using the parabolic induction for non-restricted $U_\zeta$-modules, the proof of \eqref{eq:KK} is reduced to the case where $g_{s}$ is exceptional in the sense that the semisimple rank of $G_0$ coincides with that of $G$.

Arguing similarly to \cite{BMR} using 
Theorem \ref{thm:Int2} and the split Azumaya property of 
$\DD_{\CB_\zeta,t}$ established in \cite{T1} we obtain the following result.
\begin{theorem}
\label{thm:Int3}
Assume that $\ell$ is a power of a prime.
Let $\xi$ be a character of $Z(U_\zeta)$ with associated pair  $([t],g)\in (H/W\bullet)\times G$.
We assume that $g_s$ is exceptional, $t^\ell=g_s$, $|W\bullet t|=|W|$, and 
$\ell$ is prime to $|Q/Q_0|$, where 
$Q_0$ is the root lattice of $G_0=Z_G(g_s)$.
Then we have an isomorphism
\[
K(\Mod_f(U_\zeta(\xi)))\cong K(\Mod_\coh(\CO_{\CB_0^{g_u}}))
\]
of Grothendieck groups, where 
$\CB_0$ is the flag manifold for $G_0$.
\end{theorem}
\subsection{}
The original conjecture of Lusztig in \cite{LK} is much stronger than \eqref{eq:KK}.
There, the categories 
$\Mod_f(U_\zeta(\xi))$ and
$\Mod_\coh(\CO_{\CB_0^{g_u}})$ 
are replaced by a graded counterpart and an equivariant counterpart respectively, 
and a geometric description of the basis of the equivariant $K$-group of the Springer fiber which should correspond to 
the irreducible graded $U_\zeta(\xi)$-modules is given.
To establish it we need to develop a theory 
analogous to that of \cite{BMR2}, \cite{BM}.
We will treat it in our subsequent work.
\subsection{}
A proof of  \eqref{eq:RG}, when $\ell$ is a prime greater than the Coxeter number is given in \cite{BK}; however, there is a gap in it (see also \cite[Remark 5.3]{T2}).
\subsection{}
For a ring $R$ we denote its center by $Z(R)$. 
The category of left $R$-modules is denoted as $\Mod(R)$. 
If $R$ is a commutative ring, we set $M^*=\Hom_R(M,R)$ for $M\in\Mod(R)$.

For a Hopf algebra $H$ over a commutative ring $R$ we denote the comultiplication, the counit and the antipode 
by $\Delta:H\to H\otimes H$, $\varepsilon:H\to R$, $S:H\to H$ 
respectively.
We sometimes use Sweedler's notation
\[
\Delta^{(n)}(h)=\sum_{(h)}h_{(0)}\otimes\dots\otimes h_{(n)}
\]
for the iterated comultiplication $\Delta^{(n)}:H\to H^{\otimes n+1}$.

\section{Quantized enveloping algebras}
\subsection{}
Let $\Gg$ be a finite-dimensional simple Lie algebra over the complex number field $\BC$, and let $\Gh$ be a Cartan subalgebra of $\Gg$.
Let $\Delta\subset\Gh^*$ be the root system, and let 
$W\subset GL(\Gh^*)$ be the Weyl group.
We denote by 
\begin{equation}
\label{eq:bilin}
(\;,\;):\Gh^*\times\Gh^*\to\BC
\end{equation}
the $W$-invariant symmetric bilinear form such that 
$(\alpha,\alpha)=2$ 
for short roots $\alpha$.
For $\alpha\in\Delta$ we set $\alpha^\vee=2\alpha/(\alpha,\alpha)$.
Set 
\begin{equation}
Q=\sum_{\alpha\in\Delta}\BZ\alpha,
\qquad
Q^\vee=\sum_{\alpha\in\Delta}\BZ\alpha^\vee,
\end{equation}
and
\begin{equation}
P=\{\lambda\in\Gh^*\mid
(\lambda,Q^\vee)\subset\BZ\},
\qquad
P^\vee=\{\lambda\in\Gh^*\mid
(\lambda,Q)\subset\BZ\}.
\end{equation}
Then we have
$
Q\subset Q^\vee\subset P^\vee$
and
$
Q\subset P\subset P^\vee$.
We take a set $\{\alpha_i\mid i\in I\}$ of simple roots, and denote by $\Delta^+$ the corresponding set of positive roots.
Set
\[
Q^+=\sum_{\alpha\in\Delta^+}\BZ_{\geqq0}\alpha, 
\qquad
P^+=
\{\alpha\in P\mid
(\lambda,\alpha^\vee)\geqq0\;\;(\alpha\in\Delta^+)\}.
\]
We set $a_{ij}=(\alpha_j,\alpha_i^\vee)\in\BZ$ for $i, j\in I$, and $d_i=(\alpha_i,\alpha_i)/2\in\BZ$ for $i\in I$.
Let $\Gb^+$ and $\Gb^-$ be the Borel subalgebras of $\Gg$ containing $\Gh$ such that the set of weights of $\Gb^\pm/\Gh$ coincides with $\pm\Delta^+$.
We set $\Gn^\pm=[\Gb^\pm,\Gb^\pm]$.

\subsection{}
For a $\BZ$-lattice $\Gamma$ of $\Gh^*$ satisfying 
\begin{equation}
\label{eq:Gamma}
Q\subset \Gamma\subset P^\vee
\end{equation}
we denote by $U_{\BF}(\Gg;\Gamma)$ the associative algebra over $\BF:=\BQ(q)$ generated by elements 
$k_\gamma$\;($\gamma\in \Gamma$), 
$e_i$, $f_i$\;($i\in I$)
satisfying the relations 
\begin{align}
&k_0=1,
\qquad
k_\gamma k_{\gamma'}=k_{\gamma+\gamma'}\quad&(\gamma, \gamma'\in \Gamma),
\\
&k_\gamma e_i=q^{(\alpha_i,\gamma)}e_ik_\gamma,
\quad
k_\gamma f_i=q^{-(\alpha_i,\gamma)}f_ik_\gamma
\qquad&(\gamma\in \Gamma,\;\;i\in I),
\\
&
e_if_j-f_je_i
=\delta_{ij}
\frac{k_i-k_{i}^{-1}}{q_i-q_i^{-1}}
\qquad&(i, j\in I),
\\
&
\sum_{n=0}^{1-a_{ij}}
(-1)^n
e_i^{(1-a_{ij}
-n)}e_je_i^{(n)}
=
\sum_{n=0}^{1-a_{ij}}
(-1)^n
f_i^{(1-a_{ij}
-n)}f_jf_i^{(n)}
=0
&(i,j\in I, i\ne j).
\end{align}
Here, we set $q_i=q^{d_i}$, $k_i=k_{\alpha_i}$ for $i\in I$,
and 
\[
e_i^{(n)}=e_i^n/[n]_{q_i}!,
\qquad
f_i^{(n)}=f_i^n/[n]_{q_i}!
\]
with
\[
[n]_t!=
\prod_{s=1}^n
\frac{t^{n+1-s}-t^{-n-1+s}}{t^s-t^{-s}}\in\BZ[t,t^{-1}]
\]
for $i\in I$, $n\in\BZ_{\geqq0}$.

The $\BF$-algebra $U_\BF(\Gg;\Gamma)$ is a Hopf algebra via
\begin{align}
&
\Delta(k_\gamma)=k_\gamma\otimes k_\gamma,\quad
\Delta(e_i)=e_i\otimes1+k_i\otimes e_i,\quad
\Delta(f_i)=f_i\otimes k_i^{-1}+1\otimes f_i,
\\
&
\varepsilon(k_\gamma)=1,
\quad
\varepsilon(e_i)=\varepsilon(f_i)=0,
\\
&
S(k_\gamma)=k_{-\gamma},
\quad
S(e_i)=-k_i^{-1}e_i,
\quad
S(f_i)=-f_ik_i
\end{align}
for $\gamma\in \Gamma$, $i\in I$.

We denote by $U_\BF(\Gh;\Gamma)$ 
the $\BF$-subalgebra of $U_\BF(\Gg;\Gamma)$ 
generated by 
$k_\gamma$\;($\gamma\in \Gamma$).
We denote by 
$U_\BF(\Gn^+;\Gamma)$
(resp.\
$U_\BF(\Gn^-;\Gamma)$)
the $\BF$-subalgebra of $U_\BF(\Gg;\Gamma)$ 
generated by 
$e_i$
(resp.\ $f_i$)
for $i\in I$.
We define $U_\BF(\Gb^\pm;\Gamma)$ to be the 
subalgebra of $U_\BF(\Gg;\Gamma)$ 
generated by $U_\BF(\Gh;\Gamma)$ and $U_\BF(\Gn^\pm;\Gamma)$.

\subsection{}
We set $\BA=\BZ[q,q^{-1}]$.
For $\Gamma$ as in \eqref{eq:Gamma} 
we denote by $U_\BA(\Gg;\Gamma)$ the smallest $\BA$-subalgebra of $U_\BF(\Gg;\Gamma)$ which contains 
$k_\gamma$\;($\gamma\in \Gamma$), 
$(q_i-q_i^{-1})e_i$, $(q_i-q_i^{-1})f_i$\;($i\in I$) and is stable under Lusztig's braid group action on $U_\BF(\Gg;\Gamma)$ (see \cite{DP}).
It is an $\BA$-form of $U_\BF(\Gg;\Gamma)$ in the sense that it is a free $\BA$-module satisfying 
$\BF\otimes_\BA U_\BA(\Gg;\Gamma)
\cong
U_\BF(\Gg;\Gamma)$.
It is closely related to the De Concini-Kac form of $U_\BF(\Gg;\Gamma)$, and is sometimes called the
De Concini-Procesi form.
It is a Hopf algebra over $\BA$.

Assume that $\Gamma$ is a $\BZ$-lattice of $\Gh^*$ satisfying 
\begin{equation}
\label{eq:Gamma2}
Q^\vee\subset \Gamma\subset P^\vee.
\end{equation}
In this case we define another $\BA$-form 
$U^L_\BA(\Gg;\Gamma)$
of 
$U_\BF(\Gg;\Gamma)$, called the Lusztig $\BA$-form, as follows.
Set
\[
\Gamma^*=
\{\lambda\in P\mid (\lambda,\Gamma)\subset\BZ\}.
\]
For $\lambda\in \Gamma^*$ we define an $\BF$-algebra homomorphism
\begin{equation}
\label{eq:chi}
\chi_\lambda:U_\BF(\Gh;\Gamma)\to\BF
\end{equation}
by 
$
\chi_\lambda(k_\gamma)=q^{(\lambda,\gamma)}
$ for $\gamma\in \Gamma$.
Define an $\BA$-subalgebra $U^L_\BA(\Gh;\Gamma)$ of $U_\BF(\Gh;\Gamma)$ by 
\[
U^L_\BA(\Gh;\Gamma)=
\{h\in U_\BF(\Gh;\Gamma)\mid
\chi_\lambda(h)\in\BA\;\;(\lambda\in \Gamma^*)\}.
\]
Then we define $U^L_\BA(\Gg;\Gamma)$ to be the $\BA$-subalgebra of 
$U_\BF(\Gg;\Gamma)$ generated by 
$U^L_\BA(\Gh;\Gamma)$ and 
$e_i^{(n)}$, $f_i^{(n)}$
for $i\in I$, $n\in\BZ_{\geqq0}$.
It is a Hopf algebra over $\BA$.
\begin{remark}
The Cartan part $U^L_\BA(\Gh;\Gamma)$ of our $U^L_\BA(\Gg;\Gamma)$ is slightly larger than that of the original Lusztig form (see \cite{DL}, \cite{TA}).
\end{remark}
In the rest of this paper we only use 
$U_\BA(\Gg;P)$ and $U^L_\BA(\Gg;Q^\vee)$.
So we set
\begin{equation}
U_\BA(\Gg)=U_\BA(\Gg;P),
\qquad
U^L_\BA(\Gg)
=
U^L_\BA(\Gg;Q^\vee)
\end{equation}
for simplicity.
An advantage of using $\Gamma=P$ for $U_\BA(\Gg)$ is that the structure of the center becomes simpler (see Proposition \ref{prop:HC} below).
On the other hand an 
advantage of using $\Gamma=Q^\vee$ for $U^L_\BA(\Gg)$ is that we can consider an integrable highest weight module with highest weight $\lambda$ for any $\lambda\in P^+$.

For $\Ga=\Gh, \Gn^\pm, \Gb^\pm$ we set
\[
U_\BA(\Ga)=U_\BA(\Gg)\cap U_\BF(\Ga;P),
\qquad
U^L_\BA(\Ga)=U^L_\BA(\Gg)\cap U_\BF(\Ga;Q^\vee).
\]

Let $R$ be a commutative $\BA$-algebra.
We define Hopf algebras $U_R(\Gg)$, $U^L_R(\Gg)$ by
\[
U_R(\Gg)=R\otimes_\BA U_\BA(\Gg),
\qquad
U^L_R(\Gg)=R\otimes_\BA U^L_\BA(\Gg).
\]
In particular, $U_\BF(\Gg)=U_\BF(\Gg;P)$ and 
$U^L_\BF(\Gg)=U_\BF(\Gg;Q^\vee)$.
For $\Ga=\Gh, \Gn^\pm, \Gb^\pm$ we define subalgebras 
$U_R(\Ga)$ 
(resp.\ $U^L_R(\Ga)$) 
of 
$U_R(\Gg)$ 
(resp.\ $U^L_R(\Gg)$) by
$U_R(\Ga)=R\otimes_\BA U_\BA(\Ga)$
(resp.\ 
$U^L_R(\Ga)=R\otimes_\BA U^L_\BA(\Ga)$).
We will also use
$\tilde{U}_R(\Gn^-):=S({U}_R(\Gn^-))$ and 
$\tilde{U}^L_R(\Gn^-):=S({U}^L_R(\Gn^-))$.

\subsection{}
Let $R$ be a commutative $\BA$-algebra.
Note that for $\lambda\in P$ we have an algebra homomorphism
$
\chi_\lambda:U^L_R(\Gh)\to R
$
induced by \eqref{eq:chi}.
Let $\Ga$ be one of $\Gh$, $\Gb^-$, $\Gg$.
A (left) $U^L_R(\Ga)$-module $M$ is called a weight module 
if we have
\[
M=\bigoplus_{\lambda\in P}M_\lambda,
\qquad
M_\lambda
=\{m\in M\mid
hm=\chi_\lambda(h)m\;\;(h\in U^L_R(\Gh))\}.
\]
In the case $\Ga=\Gh$ any weight module is called an integrable $U^L_R(\Gh)$-module.
In the case $\Ga=\Gg$
(resp.\
$\Ga=\Gb^-$) 
a weight module $M$ over $U^L_R(\Ga)$ is called an integrable $U^L_R(\Ga)$-module
if for any $m\in M$ there exists some $N>0$ such that we have 
$e_i^{(n)}m=f_i^{(n)}m=0$ 
(resp.\
$f_i^{(n)}m=0$) 
for $i\in I$, $n>N$.
We can similarly define the notion of an integrable right $U^L_R(\Ga)$-module.
We denote by $\Mod_\inte(U^L_R(\Ga))$ the category of integrable left $U^L_R(\Ga)$-modules.

For $\lambda\in P^+$ we define a left $U^L_\BA(\Gg)$-module $\Delta_\BA(\lambda)$ and a right
$U^L_\BA(\Gg)$-module $\Delta^r_\BA(\lambda)$
by
\begin{align*}
\Delta_\BA(\lambda)
=U^L_\BA(\Gg)/I(\lambda),
\qquad
\Delta^r_\BA(\lambda)
=U^L_\BA(\Gg)/I^r(\lambda),
\end{align*}
where $I(\lambda)$
(resp.\ $I^r(\lambda)$) 
is the left 
(resp.\ right) 
ideal of $U^L_\BA(\Gg)$ generated by 
$h-\chi_\lambda(h)$ for $h\in U^L_\BA(\Gh)$, 
$e_i^{(n)}$
(resp.\ $f_i^{(n)}$) 
for $i\in I$, $n>0$, 
and 
$f_i^{(n)}$
(resp.\ $e_i^{(n)}$) 
for $i\in I$, $n>(\lambda,\alpha_i^\vee)$.
By the theory of canonical bases they are free of finite rank over $\BA$, and their 
characters are given by Weyl's character formula.
We define a left $U^L_\BA(\Gg)$-module 
$\nabla_\BA(\lambda)$ 
and
a right $U^L_\BA(\Gg)$-module 
$\nabla^r_\BA(\lambda)$
by
\[
\nabla_\BA(\lambda)=
\Hom_\BA(\Delta^r_\BA(\lambda),\BA),
\qquad
\nabla^r_\BA(\lambda)=
\Hom_\BA(\Delta_\BA(\lambda),\BA).
\]
Here, the left (resp.\ right) action of $U^L_\BA(\Gg)$ on 
$\nabla_\BA(\lambda)$
(resp.\  $\nabla^r_\BA(\lambda)$)
is given by
\[
\langle uv^*, v\rangle
=\langle v^*, vu\rangle
\qquad
(u\in U^L_\BA(\Gg), \;
v^*\in \nabla_\BA(\lambda), \;
v\in \Delta^r_\BA(\lambda)),
\]
\[
(\text{resp.\ }
\quad
\langle v^*u, v\rangle
=\langle v^*, uv\rangle
\qquad
(u\in U^L_\BA(\Gg), \;
v^*\in \nabla^r_\BA(\lambda), \;
v\in \Delta_\BA(\lambda))
).
\]
Then $\Delta_\BA(\lambda)$ and $\nabla_\BA(\lambda)$
(resp.\
 $\Delta^r_\BA(\lambda)$ and $\nabla^r_\BA(\lambda)$)
are integrable left 
(resp.\ right)
$U^L_\BA(\Gg)$-modules.
We have a canonical injective homomorphism 
$\Delta_\BA(\lambda)\to \nabla_\BA(\lambda)$ 
(resp.\
$\Delta^r_\BA(\lambda)\to \nabla^r_\BA(\lambda)$ 
)
of left 
(resp.\ right)
$U^L_\BA(\Gg)$-modules.
For a commutative $\BA$-algebra $R$ we set
\begin{align*}
\Delta_R(\lambda)=R\otimes_\BA \Delta_\BA(\lambda),
\qquad
\nabla_R(\lambda)=R\otimes_\BA \nabla_\BA(\lambda),
\\
\Delta^r_R(\lambda)=R\otimes_\BA \Delta^r_\BA(\lambda),
\qquad
\nabla^r_R(\lambda)=R\otimes_\BA \nabla^r_\BA(\lambda).
\end{align*}

\section{Quantized coordinate algebras}
\subsection{}
We denote the connected, simply-connected simple algebraic group over $\BC$ with Lie algebra $\Gg$ by $G$.
Let $H$ (resp.\ $B^\pm$, resp.\ $N^\pm$)  be the connected closed subgroup of $G$ with Lie algebra $\Gh$ (resp.\ $\Gb^\pm$, resp.\ $\Gn^\pm$).

Let $\Ga$ be one of $\Gh$, $\Gb^-$, $\Gg$ and denote by $A$ the corresponding subgroup of $G$.
For a commutative $\BA$-algebra $R$ 
we define a Hopf algebra $\CO_R(A)$ as follows
(see \cite{APW}).
Note that $U^L_R(\Ga)^*:=\Hom_R(U^L_R(\Ga),R)$ is a $U^L_R(\Ga)$-bimodule by 
\begin{equation}
\label{eq:Ob}
\langle 
u_1\cdotp\varphi \cdotp u_2,u
\rangle
=\langle
\varphi, u_2uu_1\rangle
\qquad
(\varphi\in U^L_R(\Ga)^*, 
u, u_1, u_2\in U^L_R(\Ga)).
\end{equation}
Then $\CO_R(A)$ is defined to be the $R$-submodule of 
$U^L_R(\Ga)^*$ consisting of $\varphi\in U^L_R(\Ga)^*$ satisfying the following equivalent conditions:
\begin{itemize}
\item[(a)]
$U^L_R(\Ga)\cdotp\varphi$ is an integrable left 
$U^L_R(\Ga)$-module;
\item[(b)]
$\varphi \cdotp U^L_R(\Ga)$ is an integrable right
$U^L_R(\Ga)$-module.
\end{itemize}
Note that $\CO_R(A)$ is a $U^L_R(\Ga)$-bimodule via
\eqref{eq:Ob}.
Moreover, $\CO_R(A)$ turns out to be a Hopf algebra over $R$ by
\begin{align*}
&\langle\varphi_1\varphi_2,u\rangle
=
\langle\varphi_1\otimes\varphi_2,\Delta(u)\rangle
&(\varphi_1, \varphi_2\in\CO_R(A),\; u\in U^L_R(\Ga)),
\\
&\langle\Delta(\varphi),u_1\otimes u_2\rangle
=
\langle\varphi,u_1u_2\rangle
&
(\varphi\in\CO_R(A),\; u_1, u_2\in U^L_R(\Ga)).
\end{align*}
The inclusions 
$
U^L_R(\Gh)
\subset
U^L_R(\Gb^-)
\subset
U^L_R(\Gg)
$
induce
natural surjective Hopf algebra homomorphisms
\begin{equation}
\CO_R(G)\to\CO_R(B^-)\to\CO_R(H).
\end{equation}
In the case $\Ga=\Gg$ and $\Ga=\Gb^-$ we will sometimes regard $U_R^L(\Ga)^*$ and $\CO_R(A)$ as a left $U^L_R(\Ga)$-module via the adjoint action defined by
\begin{equation}
\label{eq:adO}
\ad(u)(\varphi)
=
\sum_{(u)}
u_{(0)}\cdotp\varphi \cdotp S^{-1}u_{(1)}
\qquad
(u\in
U^L_R(\Ga),
\;
\varphi\in U_R^L(\Ga)^*).
\end{equation}
When $\CO_R(A)$  is regarded as a left $U^L_R(\Ga)$-module via the adjoint action, we write it as $\CO_R(A)_\ad$.

\subsection{}
\label{subsec:AlgGp}
For an algebraically closed field $k$ we denote by $G_k$ the connected, simply-connected, simple algebraic group over $k$ with the same root system $\Delta$ as $G$.
We take a maximal torus $H_k$ and a Borel subgroup $B^-_k$ 
of $G_k$ such that $B_k^-\supset H_k$.
Regard $k$ as an $\BA$-algebra via $q\mapsto 1$.
Then for $\Ga=\Gg, \Gb^-, \Gh$ with $A=G, B^-, H$ respectively, the Hopf algebra $U^L_k(\Ga)$ coincides with the algebra of distributions on $A_k$.
Hence the category $\Mod_\inte(U^L_k(\Ga))$ is naturally identified with the category $\Mod(A_k)$ of (locally finite) rational $A_k$-modules.
Therefore, $\CO_k(A)$ is nothing but the algebra of regular functions on the algebraic group $A_k$ 
(see \cite[Part II, 1.20]{Ja}).

\subsection{}
Recall that for $\lambda\in P$ we have an $R$-algebra homomorphism
$\chi_\lambda:U^L_R(\Gh)\to R$.
It is easily seen that 
\begin{equation}
\label{eq:OH}
\CO_R(H)=\bigoplus_{\lambda\in P}R\chi_\lambda.
\end{equation}
We denote by $H(R)$ the set of $R$-algebra homomorphisms from $\CO_R(H)$ to $R$.

\subsection{}
We have a Hopf algebra homomorphism
$
\pi:U^L_R(\Gb^-)\to U^L_R(\Gh)
$
given by
\[
\pi(h)=h\quad
(h\in U^L_R(\Gh)),
\qquad
\pi(y)=\varepsilon(y)\quad
(y\in \tilde{U}^L_R(\Gn^-)).
\]
This induces an embedding $\CO_R(H)\hookrightarrow\CO_R(B^-)$ of Hopf algebras.
We will sometimes regard $\CO_R(H)$ 
as a Hopf subalgebra of $\CO_R(B^-)$.

Note that we have the isomorphism
\[
{U}^L_R(\Gh)
\otimes
\tilde{U}^L_R({\Gn}^-)
\cong
{U}^L_R({\Gb}^-)
\qquad(h\otimes y\leftrightarrow hy)
\]
of $R$-modules.
Note also that the $R$-algebra $\tilde{U}^L_R({\Gn}^-)$ is equipped with the natural grading 
\[
\tilde{U}^L_R({\Gn}^-)
=
\bigoplus_{\gamma\in Q^+}
\tilde{U}^L_R({\Gn}^-)_{-\gamma}
\]
such that $Sf_i^{(n)}\in \tilde{U}^L_R({\Gn}^-)_{-n\alpha_i}$
for $i\in I$, $n>0$.
We set
\begin{equation}
\tilde{U}^L_R({\Gn}^-)^\bigstar=
\bigoplus_{\gamma\in Q^+}
(\tilde{U}^L_R({\Gn}^-)_{-\gamma})^*
\subset 
\tilde{U}^L_R({\Gn}^-)^*.
\end{equation}
We will identify $\tilde{U}^L_R({\Gn}^-)^\bigstar$ with a subspace of $U^L_R(\Gb^-)^*$ by the embedding
\[
i:\tilde{U}^L_R({\Gn}^-)^\bigstar\to U^L_R(\Gb^-)^*
\]
given by
\[
\langle i(\psi),hy\rangle
=\varepsilon(h)\psi(y)
\qquad
(\psi\in \tilde{U}^L_R({\Gn}^-)^\bigstar, \;
h\in U^L_R(\Gh),\; 
y\in \tilde{U}^L_R({\Gn}^-)).
\]
The following result is standard.
\begin{proposition}
\label{prop:BHN}
\begin{itemize}
\item[(i)]
$\tilde{U}^L_R({\Gn}^-)^\bigstar$ is a subalgebra of 
$\CO_R(B^-)$ characterized by
\[
\tilde{U}^L_R({\Gn}^-)^\bigstar=
\{\psi\in\CO_R(B^-)\mid
\psi h=\varepsilon(h)\psi\;\;(h\in U^L_R(\Gh))\}.
\]

\item[(ii)]
We have 
\[
\tilde{U}^L_R({\Gn}^-)^\bigstar\otimes\CO_R(H)\cong
\CO_R(B^-)
\qquad(\psi\otimes\chi\leftrightarrow\psi\chi).
\]
Moreover, for $\psi\in\tilde{U}^L_R({\Gn}^-)^\bigstar$, $\chi\in\CO_R(H)$ we have
\[
\langle\psi\chi,hy\rangle
=
\langle\psi,y\rangle
\langle\chi,h\rangle
\qquad
(h\in U^L_R(\Gh),\; 
y\in \tilde{U}^L_R({\Gn}^-)).
\]
\end{itemize}
\end{proposition}
\subsection{}
For $\lambda\in P$ set
\[
\CO_R(G)^r_\lambda
=
\{
\varphi\in\CO_R(G)
\mid
\varphi \cdotp h=\chi_\lambda(h)\varphi
\;\;(h\in U^L_R(\Gh))\}.
\]
\begin{proposition}
\label{prop:bcO}
Let $\lambda\in P$.
Then $\CO_\BA(G)^r_\lambda$ is a free $\BA$-module.
Moreover, for any commutative $\BA$-algebra $R$ we have
$\CO_R(G)^r_\lambda\cong 
R\otimes_\BA\CO_\BA(G)^r_\lambda$.
\end{proposition}
\begin{proof}
It is proved in \cite{APW} that $\CO_R(G)^r_\lambda$ is a free $R$-module when $R$ is a local ring.
The arguments used there together with some facts on canonical bases
(see \cite[Proposition 23.3.6, Theorem 25.2.1]{Lbook})
imply our desired result.
\end{proof}

The following result is an easy consequence of 
\cite[Theorem 29.3.3]{Lbook}.
We include its proof here for the sake of readers.

\begin{proposition}
\label{prop:dualWeyl0}
As a $U_R^L(\Gg)$-bimodule, $\CO_R(G)$ has a filtration 
\begin{equation}
\label{eq:filt0}
0=C_{-1}\subset C_0\subset C_1\subset \cdots\subset \CO_R(G)
\end{equation}
such that $\CO_R(G)=\bigcup_mC_m$ and 
$C_m/C_{m-1}\cong \nabla_R(\lambda_m)\otimes \nabla^r_R(\lambda_m)$ for some $\lambda_m\in P^+$.
\end{proposition}
\begin{proof}
By Proposition \ref{prop:bcO} we have $\CO_R(G)\cong R\otimes_\BA\CO_\BA(G)$.
Hence we may assume $R=\BA$.

We fix a labelling $P^+=\{\lambda_m\mid m=0, 1, 2, \dots\}$ of $P^+$ such that 
$\lambda_m-\lambda_n\in P^+$ implies 
$m\geqq n$.
Define a two-sided ideal $I_m$ of $U^L_\BF(\Gg)$ by
\[
I_m=\{u\in U^L_\BF(\Gg)
\mid
u\Delta_\BF(\lambda_r)=0\;(r=0,\dots, m)\}.
\]
By definition we have an embedding $U^L_\BF(\Gg)/I_m\hookrightarrow\bigoplus_{r=0}^m\End_\BF(\Delta_\BF(\lambda_r))$.
On the other hand
since $\Delta_\BF(\lambda_r)$ is an irreducible 
$U^L_\BF(\Gg)/I_m$-module for $r\leqq m$, 
we see from a standard fact on finite-dimensional algebras that $\dim (U^L_\BF(\Gg)/I_m)\geqq\sum_{r=0}^m(\dim \Delta_\BF(\lambda_r))^2$.
Hence we have 
$U^L_\BF(\Gg)/I_m\cong\bigoplus_{r=0}^m\End_\BF(\Delta_\BF(\lambda_r))$.
It follows that 
\[
\CO_\BF(G)=\bigcup_m(U^L_\BF(\Gg)/I_m)^*\subset
U^L_\BF(\Gg)^*.
\]
We have also isomorphisms
\[
I_{m-1}/I_m\cong
\End_\BF(\Delta_\BF(\lambda_m))
\cong\Delta_\BF(\lambda_m)\otimes\Delta^r_\BF(\lambda_m)
\]
of two-sided $U^L_\BF(\Gg)$-modules.
Set 
$I_{\BA,m}=
I_m\cap U^L_\BA(\Gg)$ and 
$C_m=\Hom_\BA(U^L_\BA(\Gg)/I_{\BA,m},\BA)
\subset 
\Hom_\BA(U^L_\BA(\Gg),\BA)
\subset(U_\BF^L(\Gg))^*$.
Then by $\CO_\BA(G)=\CO_\BF(G)\cap\Hom_\BA(U^L_\BA(\Gg),\BA)$ we have
$
\CO_\BA(G)=\bigcup_m
C_m
$.
Hence it is sufficient to show $I_{\BA,m-1}/I_{\BA,m}\cong
\Delta_\BA(\lambda_m)\otimes\Delta_\BA^r(\lambda_m)$ as a two-sided $U_\BA^L(\Gg)$-submodule 
of
$I_{m-1}/I_{m}\cong
\Delta_\BF(\lambda_m)\otimes\Delta_\BF^r(\lambda_m)$
for any $m$.

Let us give a reformulation in terms of Lusztig's modified algebra.
Let $\dot{U}^L_\BF(\Gg)$ be the modified algebra associated to ${U}^L_\BF(\Gg)$ and let $\dot{U}^L_\BA(\Gg)$ be its natural $\BA$-form.
In the notation of \cite[Chapter 23]{Lbook} 
we have
$\dot{U}^L_\BF(\Gg)=\dot{\mathbf{U}}$ and
$\dot{U}^L_\BA(\Gg)=
\prescript{}{\mathcal{A}}
{\dot{\mathbf{U}}}$ with $v=q$.
Define a two-sided ideal $\dot{I}_m$ of $\dot{U}^L_\BF(\Gg)$ by
\[
\dot{I}_m=\{v\in \dot{U}^L_\BF(\Gg)
\mid
v\Delta_\BF(\lambda_r)=0\;(r=0,\dots, m)\},
\]
and set $\dot{I}_{\BA,m}=\dot{I}_m\cap \dot{U}^L_\BA(\Gg)$. 
Then we have 
$\dot{U}^L_\BF(\Gg)/\dot{I}_m\cong\bigoplus_{r=1}^m\End_\BF(\Delta_\BF(\lambda_r))$.
Moreover, we have 
\[
\Image(U^L_\BA(\Gg)/I_{\BA,m}\to\bigoplus_{r=1}^m\End_\BF(\Delta_\BF(\lambda_r)))
=
\Image(\dot{U}^L_\BA(\Gg)/\dot{I}_{\BA,m}\to\bigoplus_{r=1}^m\End_\BF(\Delta_\BF(\lambda_r)))
\]
for any $m$ (see the proof of Lemma 5.8 in \cite{LS}).
Hence we have
$I_{m-1}/I_m\cong \dot{I}_{m-1}/\dot{I}_m$ and 
$I_{\BA,m-1}/I_{\BA,m}\cong \dot{I}_{\BA,m-1}/\dot{I}_{\BA,m}$.
Let $\dot{J}_{m}$ be the two-sided ideal of $\dot{U}^L_\BF(\Gg)$, which is denoted by $\dot{\mathbf{U}}[\{\lambda_r\}_{r\geqq m}]$ in \cite[29.2]{Lbook}.
Since $\dot{J}_m\subset\dot{I}_m$ and 
$\dim(\dot{U}^L_\BF(\Gg)/\dot{I}_m)=
\dim(\dot{U}^L_\BF(\Gg)/\dot{J}_m)$, we have $\dot{J}_m=\dot{I}_m$.
Hence we have
\[
{I}_{m-1}/{I}_{m}\cong
\dot{\mathbf{U}}[\geq\lambda_{m}]/
\dot{\mathbf{U}}[>\lambda_m]
\]
in the notation of \cite[29.1.2]{Lbook},
and ${I}_{\BA,m-1}/{I}_{\BA,m}$ is identified with the $\BA$-submodule of 
$\dot{\mathbf{U}}[\geq\lambda_{m}]/
\dot{\mathbf{U}}[>\lambda_m]$ 
generated by $\pi(\dot{\mathbf{B}}[\lambda_m])$ in the notation of \cite[Theorem 29.3.3]{Lbook}.
Therefore, it follows from the theory of based modules (see \cite[Chapter 27]{Lbook}) and \cite[Theorem 29.3.3]{Lbook} 
that 
${I}_{\BA,m-1}/{I}_{\BA,m}$ is generated as a two-sided $U^L_\BA(\Gg)$-submodule of 
$I_{m-1}/I_m\cong\Delta_\BF(\lambda_m)\otimes\Delta^r_\BF(\lambda_m)$ 
by a non-zero element $v$ satisfying 
$e_iv=ve_i=0$ for any $i\in I$.
It follows that
${I}_{\BA,m-1}/{I}_{\BA,m}
\cong
\Delta_\BA(\lambda_m)\otimes\Delta^r_\BA(\lambda_m)
$.
\end{proof}
From this we readily obtain the following.
This fact is crucial in our argument below.
\begin{proposition}
\label{prop:dualWeyl1}
As a $U_R^L(\Gg)$-module with respect to the adjoint action, $\CO_R(G)_\ad$ has a filtration 
\begin{equation}
\label{eq:filt}
0=L_{-1}\subset L_0\subset L_1\subset \cdots\subset \CO_R(G)_\ad
\end{equation}
such that $\CO_R(G)_\ad=\bigcup_mL_m$ and 
$L_m/L_{m-1}\cong \nabla_R(\mu_m)$ for some $\mu_m\in P^+$.
\end{proposition}
\begin{proof}
The filtration \eqref{eq:filt0} in 
Proposition \ref{prop:dualWeyl0} is stable under the adjoint action, and we have
\[
C_m/C_{m-1}\cong
\nabla_R(\lambda_m)\otimes\nabla_R(-w_0\lambda_m)
\]
with respect to the adjoint action,
where $w_0$ is the longest element of $W$, and 
the $U_R^L(\Gg)\otimes U_R^L(\Gg)$-module
$\nabla_R(\lambda_m)\otimes
\nabla_R(-w_0\lambda_m)$ is regarded as a $U_R^L(\Gg)$-module via the comultiplication 
$\Delta:U_R^L(\Gg)
\to
U_R^L(\Gg)\otimes U_R^L(\Gg)$.
Hence by \cite[27.3.3]{Lbook} and \cite[3.3 (v)]{Xi}
we obtain the desired filtration
\eqref{eq:filt}
by refining \eqref{eq:filt0}.
\end{proof}

\section{Induction functor}
\subsection{}
Assume
\begin{equation}
\label{eq:ac}
\text{$(\Ga,\Gc)$ is one of 
$(\Gg,\Gb^-)$, $(\Gb^-,\Gh)$, $(\Gg,\Gh)$,}
\end{equation}
and let $A$ be the connected closed subgroup of $G$ corresponding to $\Ga$.

Let $R$ be a commutative $\BA$-algebra.
Following \cite{APW} we define a left exact functor 
\begin{equation}
\Ind^{\Ga,\Gc}_{R}
:
\Mod_\inte(U^L_R(\Gc))
\to
\Mod_\inte(U^L_R(\Ga))
\end{equation}
as follows.
Let $M\in \Mod_\inte(U^L_R(\Gc))$.
Define a left $U^L_R(\Gc)$-module structure of 
 $\CO_R(A)\otimes_R M$
 by
\[
y\star(\varphi\otimes m)
=\sum_{(y)}\varphi\cdotp S^{-1}y_{(0)}\otimes y_{(1)}m
\qquad
(y\in U^L_R(\Gc), \; \varphi\in\CO_R(A),\; m\in M)
\]
and set 
\[
\Ind^{\Ga,\Gc}_{R}(M)
:=
\{z\in\CO_R(A)\otimes M\mid
y\star z=\varepsilon(y)z\;\;(y\in U^L_R(\Gc))\}.
\]
Then the left $U^L_R(\Ga)$-module structure of 
$\CO_R(A)\otimes M$ given by 
\[
u(\varphi\otimes m)
=u\cdotp \varphi\otimes m
\qquad
(u\in U^L_R(\Ga), \;  \varphi\in\CO_R(A),\; m\in M)
\]
induces a left $U^L_R(\Ga)$-module structure of 
$\Ind^{\Ga,\Gc}_{R}(M)$ so that
$\Ind^{\Ga,\Gc}_{R}(M)\in\Mod_\inte(U^L_R(\Ga))$.

We recall some  results in \cite{APW} in the following.
For the convenience of the readers we include some of their proofs.
\begin{proposition}[Frobenius reciprocity]
\label{prop:Frob}
For $V\in\Mod_\inte(U^L_R(\Ga))$, $M\in\Mod_\inte(U^L_R(\Gc))$
we have
\[
\Hom_{U^L_R(\Ga)}(V,\Ind^{\Ga,\Gc}_{R}(M))
\cong
\Hom_{U^L_R(\Gc)}(V,M).
\]
\end{proposition}
For $g\in \Hom_{U^L_R(\Gc)}(V,M)$ the corresponding 
$g'\in \Hom_{U^L_R(\Ga)}(V,\Ind^{\Ga,\Gc}_{R}(M))$
is given by
\[
g'(v)=\sum_j\varphi_j\otimes g(v_j)
\qquad(v\in V),
\]
for $\varphi_j\in\CO_R(A)$, $v_j\in V$ such that
\[
uv=\sum_j\langle\varphi_j,u\rangle v_j
\qquad(u\in U^L_R(\Ga)).
\]

It follows from Proposition  \ref{prop:Frob} that the following transitivity holds
\begin{equation}
\Ind^{\Gg,\Gb^-}_{R}\circ
\Ind^{\Gb^-,\Gh}_{R}
=
\Ind^{\Gg,\Gh}_{R}.
\end{equation}
\begin{lemma}
\label{lem:FG}
For $V\in\Mod_\inte(U^L_R(\Ga))$
there exist canonical homomorphisms
\[
\CF\in
\Hom_{U^L_R(\Ga)}(V,\Ind^{\Ga,\Gc}_{R}(V)),
\qquad
\CE\in
\Hom_{U^L_R(\Gc)}(
\Ind^{\Ga,\Gc}_{R}(V),V)
\]
satisfying $\CE\circ\CF=\id$.
\end{lemma}
\begin{proof}
We can take $\CF$ and $\CE$ to be 
the homomorphism corresponding to 
$\id\in\Hom_{U_R(\Gh)}(V,V)$ under the Frobenius reciprocity, and 
the restriction of 
$
\varepsilon\otimes\id:\CO_R(A)\otimes V\to V
$
to $\Ind^{\Ga,\Gc}_{R}(V)$ respectively.
\end{proof}

\subsection{}
Let $\Ga=\Gg$ or $\Gb^-$.
\begin{lemma}
\label{lem:from h}
\begin{itemize}
\item[(i)]
$\Ind^{\Ga,\Gh}_{R}$ is an exact functor.
\item[(ii)]
Let $M\in\Mod_\inte(U^L_R(\Gh))$.
If any weight space of $M$ is a free $($resp.\ projective, resp.\ flat$)$ $R$-module, then any weight space of 
$\Ind^{\Ga,\Gh}_{R}(M)$ is a
free $($resp.\ projective, resp.\ flat$)$ $R$-module.
\end{itemize}
\end{lemma}
\begin{proof}
By definition we have
\[
\Ind^{\Ga,\Gh}_{R}(M)
=
\bigoplus_{\lambda\in P}
\CO_R(A)^r_\lambda\otimes M_\lambda,
\]
where 
\[
\CO_R(A)^r_\lambda=
\{\varphi\in\CO_R(A)\mid
\varphi \cdotp h=\chi_\lambda(h)\varphi
\;\;(h\in U_R(\Gh))\}.
\]
Hence it is sufficient to show that $\CO_R(A)^r_\lambda$ is a free $R$-module for any $\lambda\in P$.
In the case $A=G$ this can be seen from Proposition \ref{prop:bcO}.
In the case $A=B^-$ this is an easy consequence of  Proposition \ref{prop:BHN}.
\end{proof}

Let $M\in\Mod_\inte(U^L_R(\Ga))$.
In \cite{APW}  a canonical resolution 
\begin{equation}
M\to Q^\bullet
\end{equation}
in $\Mod_\inte(U^L_R(\Ga))$
called the standard resolution
was introduced.
It will play a crucial role in our argument below.
Let us recall its definition.
Define  $\CF\in \Hom_{U_R(\Ga)}(M,\Ind^{\Ga,\Gh}_{R}(M))$ as in Lemma \ref{lem:FG}.
By Lemma  \ref{lem:FG} $\CF$ is injective, and we obtain 
\[
M\hookrightarrow Q^0:=\Ind^{\Ga,\Gh}_{R}(M).
\]
Applying the above argument to 
$\tilde{Q}^1=\Cok(M\to Q^0)\in\Mod_\inte(U^L_R(\Ga))$ 
we obtain 
\[
\tilde{Q}^1\hookrightarrow Q^1:=\Ind^{\Ga,\Gh}_{R}(\tilde{Q}^1).
\]
Repeating this we obtain an exact sequence 
\[
0\to M\to Q^0\to Q^1\to\cdots
\]
in $\Mod_\inte(U^L_R(\Ga))$ 
such that 
\[
Q^j=\Ind^{\Ga,\Gh}_{R}(\tilde{Q}^j)\in \Mod_\inte(U^L_R(\Ga)),
\quad
\tilde{Q}^j=\Cok(\tilde{Q}^{j-1}\to{Q}^{j-1})
\in \Mod_\inte(U^L_R(\Ga))
\]
with $\tilde{Q}^0=M$.

\begin{lemma}
\label{lem:STD}
Let $M\in\Mod_\inte(U^L_R(\Ga))$, and let 
$M\to Q^\bullet$ be its standard resolution
with $Q^j=\Ind^{\Ga,\Gh}_{R}(\tilde{Q}^j)$.
If any weight space of $M$ is a  projective
$($resp.\ flat$)$
$R$-module, then any weight space of $\tilde{Q}^j$ and ${Q}^j$ is a
projective 
$($resp.\ flat$)$
$R$-module.
\end{lemma}
\begin{proof}
By Lemma \ref{lem:from h} any weight space of $Q^0$ is projective
(resp.\ flat).
By Lemma \ref{lem:FG} $M\hookrightarrow Q^0$ is a split homomorphism of $U_R(\Gh)$-modules, and hence any weight space of $\tilde{Q}^1$ is projective
(resp.\ flat).
It follows from Lemma \ref{lem:from h} that 
any weight space of ${Q}^1$ is projective
(resp.\ flat).
Using the same argument we deduce the same property for any $j>1$ by induction on $j$.
\end{proof}

\subsection{}
By the Frobenius reciprocity we see easily the following.
\begin{lemma}
\label{lem:II}
Let $(\Ga,\Gc)$ be as in \eqref{eq:ac}.
Then $\Ind^{\Ga,\Gc}_{R}$ sends injective objects to injective objects.
\end{lemma}
\begin{lemma}
\label{lem:EI}
Let $\Ga$ be one of $\Gg$, $\Gb^-$, $\Gh$.
Then $\Mod_\inte(U^L_R(\Ga))$ has enough injectives.
\end{lemma}
\begin{proof}
In the case $\Ga=\Gh$ this follows easily from the fact that 
$\Mod(R)$ has enough injectives.
Assume $\Ga=\Gg$ or $\Gb^-$.
Let $M\in\Mod_\inte(U^L_R(\Ga))$.
There exists an injective object $I$ of $\Mod_\inte(U^L_R(\Gh))$ and an embedding 
$M\hookrightarrow I$ in $\Mod_\inte(U^L_R(\Gh))$.
By Lemma \ref{lem:FG} we obtain an embedding
\[
M\hookrightarrow
\Ind^{\Ga,\Gh}_{R}(M)\hookrightarrow
\Ind^{\Ga,\Gh}_{R}(I)
\]
in  $\Mod_\inte(U^L_R(\Ga))$.
Then $\Ind^{\Ga,\Gh}_{R}(I)$ is an injective object of 
$\Mod_\inte(U^L_R(\Ga))$ by Lemma \ref{lem:II}.
\end{proof}

It follows that we have right derived functors
\begin{equation}
R^j\Ind^{\Ga,\Gc}_{R}:\Mod_\inte(U^L_R(\Gc))
\to
\Mod_\inte(U^L_R(\Ga))
\qquad(j\geqq0)
\end{equation}
for $(\Ga,\Gc)$ as in \eqref{eq:ac}.
By Lemma \ref{lem:from h} (i) we have
$R^j\Ind^{\Ga,\Gh}_{R}=0$ for $j>0$ and $\Ga=\Gg$ or $\Gb^-$.
\begin{lemma}
Let $M\in\Mod_\inte(U^L_R(\Gb^-))$ and let $M\to Q^\bullet$ be its standard resolution.
Then we have
\[
R^j\Ind^{\Gg,\Gb^-}_{R}(M)=
H^j(\Ind^{\Gg,\Gb^-}_R(Q^\bullet))
\qquad (j\geqq0).
\]
\end{lemma}
\begin{proof}
It is sufficient to show
\[
R^j\Ind^{\Gg,\Gb^-}_{R}(Q^k)=0
\qquad(j>0, \;k\geqq0).
\]
Recall $Q^k=\Ind^{\Gb^-,\Gh}_{R}(\tilde{Q}^k)$
for some $\tilde{Q}^k\in\Mod_\inte(U_R(\Gh))$.
Hence we have
\[
R^j\Ind^{\Gg,\Gb^-}_{R}(Q^k)
=
R^j\Ind^{\Gg,\Gh}_{R}(\tilde{Q}^k)=0.
\]
\end{proof}
\subsection{}
Let $R\to S$ be a homomorphism of commutative $\BA$-algebras.

Let $\Ga$ be one of $\Gg$, $\Gb^-$, $\Gh$.
We denote by
\begin{equation}
\For:
\Mod_\inte(U^L_S(\Ga))\to
\Mod_\inte(U^L_R(\Ga))
\end{equation}
the forgetful functor.
We have also a right exact functor
\begin{equation}
S\otimes_R(\bullet):
\Mod_\inte(U^L_R(\Ga))\to
\Mod_\inte(U^L_S(\Ga)).
\end{equation}

We have
\begin{equation}
\label{eq:BCO}
S\otimes_R\CO_R(A)\cong \CO_S(A)
\end{equation}
by  \eqref{eq:OH}, Proposition \ref{prop:BHN}, Proposition \ref{prop:bcO}.
The following is easily checked from this and from the definition of $\Ind$.
\begin{lemma}
\label{lem:INDfor}
Let $(\Ga,\Gc)$ be as in \eqref{eq:ac}.
Then for $M\in \Mod_\inte(U^L_S(\Gc))$ we have
\[
\For(\Ind^{\Ga,\Gc}_{S}(M))
\cong
\Ind^{\Ga,\Gc}_{R}(\For(M)).
\]
\end{lemma}
Hence we have the following.
\begin{lemma}
\label{lem:STDfor}
Let $\Ga=\Gg$ or $\Gb^-$.
Let $M\in \Mod_\inte(U^L_S(\Ga))$, and 
let $M\to Q^\bullet$ be its standard resolution.
Then $\For(M)\to \For(Q^\bullet)$ is the standard resolution of 
$\For(M)\in \Mod_\inte(U^L_R(\Ga))$.
\end{lemma}
From this we obtain the following.
\begin{proposition}
Let $(\Ga,\Gc)$ be as in \eqref{eq:ac}.
Then for $M\in \Mod_\inte(U^L_S(\Gc))$ we have
\[
\For(R^i\Ind^{\Ga,\Gc}_{S}(M))
\cong
R^i\Ind^{\Ga,\Gc}_{R}(\For(M)).
\]
\end{proposition}

The following is easily checked from the definition of $\Ind$ (see the proof of Lemma \ref{lem:from h}).
\begin{lemma}
\label{lem:INDbc}
Let $\Ga=\Gg$ of $\Gb^-$.
Then for $M\in \Mod_\inte(U^L_R(\Gh))$ we have
\[
S\otimes_R\Ind^{\Ga,\Gh}_{R}(M)
\cong
\Ind^{\Ga,\Gh}_{S}(S\otimes_RM).
\]
\end{lemma}
Hence we have the following.
\begin{lemma}
\label{lem:STDbc}
Let $\Ga=\Gg$ of $\Gb^-$.
Let $M\in \Mod_\inte(U^L_R(\Ga))$, and 
let $M\to Q^\bullet$ be its standard resolution.
Then $S\otimes_R M\to S\otimes_RQ^\bullet$ is the standard resolution of 
$S\otimes_R M\in \Mod_\inte(U^L_S(\Ga))$.
\end{lemma}
From this we obtain the following.
\begin{proposition}
\label{prop:APW0}
Assume that $S$ is flat over $R$.
Then for $M\in \Mod_\inte(U^L_R(\Gb^-))$ we have
\[
S\otimes_RR^j\Ind^{\Gg,\Gb^-}_{R}(M)
\cong
R^j\Ind^{\Gg,\Gb^-}_{S}(S\otimes_RM)
\qquad(j\geqq0).
\]
\end{proposition}

We will also need the following (see \cite[3.5]{APW}).
\begin{proposition}
\label{prop:APW1}
Assume 
\begin{equation}
\label{eq:APW1}
\Tor^R_k(S,E)=0\qquad(k\geqq2)
\end{equation}
for any $R$-module $E$.
Let 
$M\in\Mod_\inte(U_R(\Gb^-))$.
Assume that 
any weight space of $M$ is a flat $R$-module.
Then  the canonical homomorphism
\[
S\otimes_R R^j\Ind^{\Gg,\Gb^-}_{R}(M)
\to
R^j\Ind^{\Gg,\Gb^-}_{S}(S\otimes_RM)
\qquad(j\geqq0)
\]
in $\Mod_\inte(U^L_S(\Gg))$ is injective.
\end{proposition}
\begin{proof}
Let $M\to Q^\bullet$ be the standard resolution of $M$, and set
$
N^\bullet=\Ind^{\Gg,\Gb^-}_{R}(Q^\bullet)
$.
Then 
$N^\bullet$
is a complex in 
$\Mod_\inte(U^L_R(\Gg))$ satisfying 
\[
H^j(N^\bullet)=R^j\Ind^{\Gg,\Gb^-}_{R}(M).
\]
Denote by $d^j:N^j\to N^{j+1}$ the differential  of 
$N^\bullet$ and set
\[
B^j=\Image(d^j),
\qquad
C^j=\Cok(d^j).
\]
Then we have exact sequences
\begin{align}
\label{eq:ex1}
&0\to B^j\to N^{j+1}\to C^j\to0
\qquad(j\geqq -1),
\\
\label{eq:ex2}
&0\to H^j(N^\bullet)\to C^{j-1}\to B^j\to 0 
\qquad(j\geqq0).
\end{align}
Here, $B^{-1}=0$, $C^{-1}=N^0$.
By
\[
N^j=
\Ind^{\Gg,\Gb^-}_{R}(Q^j)
\cong
\Ind^{\Gg,\Gh}_{R}(\tilde{Q}^j)
\]
we see from Lemma \ref{lem:from h} and 
Lemma \ref{lem:STD} that
any weight space of $N^j$ is a flat $R$-module.
Hence by \eqref{eq:ex1}
we obtain
\begin{align*}
&\Tor_k^R(S,C^j)\cong
\Tor_{k-1}^R(S,B^j)\qquad(k\geqq2,\;\;j\geqq0).
\end{align*}
This together with our assumption \eqref{eq:APW1} implies 
\[
\Tor_{k}^R(S,B^j)=0\quad(j\geqq0,\;\;k\geqq1).
\]
Therefore, it follows from 
\eqref{eq:ex2} that we have an exact sequence
\begin{align}
\label{eq:ex2a}
&0\to S\otimes_RH^j(N^\bullet)
\to S\otimes_RC^{j-1}
\to S\otimes_RB^j
\to 0.
\end{align}

On the other hand by Lemma \ref{lem:STDbc}
$S\otimes_RM\to S\otimes_RQ^\bullet$
is the standard resolution of 
$S\otimes_RM$, and hence
\[
R^j\Ind^{\Gg,\Gb^-}_{S}(S\otimes_RM)
\cong
H^j(\Ind^{\Gg,\Gb^-}_{S}(S\otimes_RQ^\bullet)).
\]
Moreover, by Lemma \ref{lem:INDbc} we have
\begin{align*}
\Ind^{\Gg,\Gb^-}_{S}(S\otimes_RQ^j)
\cong&
\Ind^{\Gg,\Gh}_{S}(S\otimes_R\tilde{Q}^j)
\cong
S\otimes_R\Ind^{\Gg,\Gh}_{R}(\tilde{Q}^j)
\cong
S\otimes_R\Ind^{\Gg,\Gb^-}_{R}({Q}^j)
\\
=&
S\otimes_RN^j.
\end{align*}
Hence we have
\[
R^j\Ind^{\Gg,\Gb^-}_{S}(S\otimes_RM)=
H^j(S\otimes_RN^\bullet).
\]
Now we apply 
the above argument for $N^\bullet$ to $S\otimes_RN^\bullet$.
Let $\tilde{d}^j:S\otimes_RN^j\to S\otimes_RN^{j+1}$ be the differential of 
$S\otimes_RN^\bullet$, and set
\[
\tilde{B}^j=\Image(\tilde{d}^j),
\qquad
\tilde{C}^j=\Cok(\tilde{d}^j).
\]
Then we have
\begin{equation}
\label{eq:Eex2}
0\to H^j(S\otimes_RN^\bullet)\to \tilde{C}^{j-1}\to\tilde{B}^j\to 0.
\end{equation}
Note that
\[
\tilde{C}^j=\Cok(\tilde{d}^j)\cong
S\otimes_R\Cok(d^j)=
S\otimes_RC^j.
\]
Note also that the canonical homomorphism
\[
S\otimes_RB^j
=
S\otimes_R\Image(d^j)
\to
\Image(\tilde{d}^j)
=
\tilde{B}_j
\]
is surjective.
Hence our assertion follows from 
\eqref{eq:ex2a}, \eqref{eq:Eex2}.
\end{proof}

\begin{proposition}
\label{prop:sp}
Assume that there exists some $N>0$ such that 
\begin{equation}
\label{eq:sp}
\Tor^R_k(S,E)=0\qquad(k>N)
\end{equation}
for any $R$-module $E$.
Let 
$M\in\Mod_\inte(U^L_R(\Gb^-))$.
Assume that 
any weight space of $M$ is a flat $R$-module
and
we have
$R^j\Ind^{\Gg,\Gb^-}_{R}(M)=0$ for any $j>0$.
Then  we have
$R^j\Ind^{\Gg,\Gb^-}_{S}(S\otimes_RM)=0$ for any $j>0$, and 
the canonical homomorphism
\begin{equation}
\label{eq:sp2}
S\otimes_R\Ind^{\Gg,\Gb^-}_{R}(M)
\to
\Ind^{\Gg,\Gb^-}_{S}(S\otimes_RM)
\end{equation}
is an isomorphism
in $\Mod_\inte(U^L_S(\Gg))$.
\end{proposition}
\begin{proof}
Set $\tilde{M}=\Ind^{\Gb^-,\Gh}_{R}(M)$, and 
let $M\hookrightarrow \tilde{M}$ be the canonical embedding in $\Mod_\inte(U^L_R(\Gb^-))$.
We first show that $\tilde{M}/M$ satisfies the same assumption as $M$.
By Lemma \ref{lem:from h} any weight space of $\tilde{M}$ is a flat $R$-module.
By Lemma \ref{lem:FG} 
$M\hookrightarrow \tilde{M}$ splits as a homomorphism of $U^L_R(\Gh)$-modules, and hence any weight space of $\tilde{M}/M$ is also a flat $R$-module.
Moreover, by Lemma \ref{lem:from h} we have
\[
R^j\Ind^{\Gg,\Gb^-}_{R}(\tilde{M})
=R^j\Ind^{\Gg,\Gh}_{R}(M)=0
\qquad(j>0).
\]
Hence 
we obtain
\[
R^j\Ind^{\Gg,\Gb^-}_{R}(\tilde{M}/M)
=0
\qquad(j>0)
\]
from the exact sequence 
\begin{equation}
\label{eq:exI}
0\to M\to\tilde{M}\to \tilde{M}/M\to 0
\end{equation}
and the assumption on $M$.
We have proved that  $\tilde{M}/M$ satisfies the same assumption as $M$.

By the above argument we also obtain an exact sequence
\begin{equation}
\label{eq:exII}
0\to \Ind^{\Gg,\Gb^-}_{R}(M)
\to \Ind^{\Gg,\Gb^-}_{R}(\tilde{M})
\to \Ind^{\Gg,\Gb^-}_{R}(\tilde{M}/M)\to 0.
\end{equation}
Since \eqref{eq:exI} is a split exact sequence of $R$-modules, we have an exact sequence
\begin{equation}
\label{eq:exIII}
0\to S\otimes_RM\to
S\otimes_R\tilde{M}\to 
S\otimes_R\tilde{M}/M\to 0
\end{equation}
in $\Mod_\inte(U^L_S(\Gb^-))$.

Let us show that \eqref{eq:sp2} is an isomorphism.
By $\Ind^{\Gg,\Gb^-}_{R}(\tilde{M})\cong
\Ind^{\Gg,\Gh}_{R}({M})$ any weight space 
of 
$\Ind^{\Gg,\Gb^-}_{R}(\tilde{M})$ is a flat $R$-module.
Hence from \eqref{eq:exII} we obtain
\begin{equation}
\Tor_{k+1}^R(S, \Ind^{\Gg,\Gb^-}_{R}(\tilde{M}/M))
\cong
\Tor_{k}^R(S, \Ind^{\Gg,\Gb^-}_{R}(M))
\qquad(k\geqq1).
\end{equation}
Since $\tilde{M}/M$ satisfies the same assumption as $M$, we see by the backward induction on $k$ that 
$\Tor_{k}^R(S, \Ind^{\Gg,\Gb^-}_{R}(M))=0$ for any $k>0$ (here, we used \eqref{eq:sp}).
We obtain an exact sequence 
\begin{equation}
0\to
S\otimes_R\Ind^{\Gg,\Gb^-}_{R}(M)
\to 
S\otimes_R\Ind^{\Gg,\Gb^-}_{R}(\tilde{M})
\to 
S\otimes_R\Ind^{\Gg,\Gb^-}_{R}(\tilde{M}/M)\to 0.
\end{equation}
Now consider the following commutative diagram whose rows are exact
\[
\xymatrix@C=17pt{
0
\ar[r]
&
S\otimes_R\Ind^{\Gg,\Gb^-}_{R}(M)
\ar[r]
\ar[d]
&
S\otimes_R\Ind^{\Gg,\Gb^-}_{R}(\tilde{M})
\ar[r]
\ar[d]
&
S\otimes_R\Ind^{\Gg,\Gb^-}_{R}(\tilde{M}/M)
\ar[r]
\ar[d]
&
0
\\
0
\ar[r]
&
\Ind^{\Gg,\Gb^-}_{S}(S\otimes_RM)
\ar[r]
&
\Ind^{\Gg,\Gb^-}_{S}(S\otimes_R\tilde{M})
\ar[r]
&
\Ind^{\Gg,\Gb^-}_{S}(S\otimes_R\tilde{M}/M).
}
\]
%
The middle vertical arrow is an isomorphism by
\[
S\otimes_R\Ind^{\Gg,\Gb^-}_{R}(\tilde{M})
\cong
S\otimes_R\Ind^{\Gg,\Gh}_{R}({M})
\cong
\Ind^{\Gg,\Gh}_{S}(S\otimes_R{M})
\cong
\Ind^{\Gg,\Gb^-}_{S}(S\otimes_R\tilde{M}).
\]
Hence the leftmost arrow is injective.
Since $\tilde{M}/M$ satisfies the same assumption as $M$, the rightmost arrow is also injective.
It follows that the leftmost arrow is bijective.
Namely, \eqref{eq:sp2} is an isomorphism.
Note that the rightmost arrow is an isomorphism by the same reason, and hence we have also shown that 
$\Ind^{\Gg,\Gb^-}_{S}(S\otimes_R\tilde{M})
\to
\Ind^{\Gg,\Gb^-}_{S}(S\otimes_R\tilde{M}/M)
$
is surjective.

It remains to show 
$R^j\Ind^{\Gg,\Gb^-}_{S}(S\otimes_R{M})=0$ for $j>0$.
By 
\[
R^j\Ind^{\Gg,\Gb^-}_{S}(S\otimes_R\tilde{M})
=
R^j\Ind^{\Gg,\Gh}_{S}(S\otimes_R{M})=0
\qquad(j>0)
\]
we obtain from \eqref{eq:exIII} and the surjectivity of 
$\Ind^{\Gg,\Gb^-}_{S}(S\otimes_R\tilde{M})
\to
\Ind^{\Gg,\Gb^-}_{S}(S\otimes_R\tilde{M}/M)
$ that
\begin{equation}
R^1\Ind^{\Gg,\Gb^-}_{S}(S\otimes_R{M})=0,
\end{equation}
\begin{equation}
R^{j+1}\Ind^{\Gg,\Gb^-}_{S}(S\otimes_R{M})
\cong
R^{j}\Ind^{\Gg,\Gb^-}_{S}(S\otimes_R\tilde{M}/M)
\qquad(j>0).
\end{equation}
Since $\tilde{M}/M$ satisfies the same assumption as $M$, we obtain the desired result by induction on $j$.
\end{proof}

\section{$\Ext$ functor}
\subsection{}
Let $\Ga=\Gg$ or $\Gb^-$.
For $V\in \Mod_\inte(U^L_R(\Ga))$  we denote the $j$-th right derived functor of the left exact functor
\[
\Hom_{R,\Ga}(V,\bullet)
:=
\Hom_{U_R(\Ga)}(V,\bullet)
: \Mod_\inte(U^L_R(\Ga))
\to
\Mod(R)
\]
by
\[
\Ext^j_{R,\Ga}(V,\bullet)
: \Mod_\inte(U^L_R(\Ga))
\to
\Mod(R).
\]

Let $R\to S$ be a homomorphism of commutative $\BA$-algebras.
We can prove the following results 
similarly to Proposition \ref{prop:APW0} and 
Proposition \ref{prop:APW1}.
\begin{proposition}
\label{prop:APW2}
Assume that $S$ is flat over $R$.
Let $V, M\in \Mod_\inte(U^L_R(\Ga))$.
Assume that $V$ is a finitely generated projective $R$-module.
Then for any $j\geqq0$ we have
\[
S\otimes_R
\Ext^j_{R,\Ga}(V,M)
\cong
\Ext^j_{S,\Ga}(S\otimes_RV,S\otimes_RM).
\]
\end{proposition}
\begin{proposition}
\label{prop:APW3}
Assume \eqref{eq:APW1} holds for any $R$-module $E$.
Let $V, M\in \Mod_\inte(U^L_R(\Ga))$.
Assume that $V$ is a finitely generated projective $R$-module.
Assume also that 
any weight space of $M$ is a flat $R$-module.
Then the canonical homomorphism
\[
S\otimes_R
\Ext^j_{R,\Ga}(V,M)
\to
\Ext^j_{S,\Ga}(S\otimes_RV,S\otimes_RM)
\]
is injective for any $j\geqq0$.
\end{proposition}

\subsection{}
\begin{proposition}
\label{prop:dualWeyl2}
\begin{itemize}
\item[(i)]
We have
\[
\Ext^j_{R,\Gg}(\Delta_R(\lambda),
\CO_R(G)_\ad)=0
\qquad
(\lambda\in P^+,\; j>0).
\]
\item[(ii)]
Let $R\to S$ be a homomorphism of commutative $\BA$-algebras.
Then we have
\[
S\otimes_R
\Hom_{R,\Gg}(\Delta_R(\lambda),\CO_R(G)_\ad)
\cong
\Hom_{S,\Gg}(\Delta_S(\lambda),\CO_S(G)_\ad)
\qquad
(\lambda\in P^+).
\]
\end{itemize}
\end{proposition}
\begin{proof}
For $\lambda, \mu\in P^+$ and $j\geqq0$ we have
\begin{equation}
\label{eq:CPSK}
\Ext^j_{R,\Gg}(\Delta_R(\lambda),\nabla_R(\mu))
=
\begin{cases}
R\qquad&(\lambda=\mu,\;j=0)
\\
0&(\text{otherwise})
\end{cases}
\end{equation}
(see \cite{And}).
Take a filtration \eqref{eq:filt} as in Proposition \ref{prop:dualWeyl1}.

(i) Let $\lambda\in P^+$ and $j>0$.
It is sufficient to show 
\[
\Ext^j_{R,\Gg}(\Delta_R(\lambda),L_m)=0
\]
for any $m$.
This follows from \eqref{eq:CPSK} and the exact sequence 
\[
\Ext^j_{R,\Gg}(\Delta_R(\lambda),L_{m-1})
\to
\Ext^j_{R,\Gg}(\Delta_R(\lambda),L_m)
\to
\Ext^j_{R,\Gg}(\Delta_R(\lambda),L_m/L_{m-1})
\]
using induction on $m$.

(ii) Let $\lambda\in X^+$.
It is sufficient to show 
\[
S\otimes_R\Hom_{R,\Gg}(\Delta_R(\lambda),L_m)
\cong
\Hom_{S,\Gg}(\Delta_S(\lambda),S\otimes_RL_m)
\]
for any $m$.
By \eqref{eq:CPSK} we have
\[
S\otimes_R\Hom_{R,\Gg}(\Delta_R(\lambda),L_m/L_{m-1})
\cong
\Hom_{S,\Gg}(\Delta_S(\lambda),S\otimes_RL_m/S\otimes_RL_{m-1}),
\]
and hence our assertion follows by induction on $m$.
\end{proof}
\subsection{}
Recall that the surjective Hopf algebra homomorphism
$\pi:U^L_R(\Gb^-)\to U^L_R(\Gh)$ induces an embedding
$\CO_R(H)\to \CO_R(B^-)$ of Hopf algebras.
We will regard $\CO_R(H)$ as a Hopf subalgebra of $\CO_R(B^-)$ in the following.
We also regard $\CO_R(B^-)$ as an $\CO_R(H)$-module via the right multiplication.
\begin{lemma}
\label{lem:seB}
The adjoint action of $U^L_R(\Gb^-)$ on $\CO_R(B^-)$ is 
$\CO_R(H)$-linear.
\end{lemma}
\begin{proof}
For $\varphi, \psi\in\CO_R(B^-)$ and $u\in U^L_R(\Gb^-)$  we have
\[
\ad(u)(\varphi\psi)=
\sum_{(u)_2}
(u_{(0)}\cdotp \varphi\cdotp S^{-1}u_{(2)})
(\ad(u_{(1)})(\psi)).
\]
If $\psi\in\CO_R(H)$, we have $\ad(u)(\psi)=\varepsilon(u)\psi$ for any $u\in  U^L_R(\Gb^-)$.
Hence for $\varphi\in\CO_R(G)$,  $\psi\in\CO_R(H)$,  $u\in  U^L_R(\Gb^-)$ we have
\[
\ad(u)(\varphi\psi)=
\sum_{(u)}
\varepsilon(u_{(1)})
(u_{(0)}\cdotp \varphi\cdotp S^{-1}u_{(2)})
\psi
=(\ad(u)(\varphi))\psi.
\]
\end{proof}
Hence we can regard $\CO_R(B^-)_\ad$ as an object of
$\Mod_\inte(U^L_{\CO_R(H)}(\Gb^-))$.

The aim of this subsection is to show the following.
\begin{proposition}
\label{prop:FG}
Assume that $R$ is a commutative Noetherian $\BA$-algebra.
Then for $\lambda\in P^+$ and $j\geqq0$
the $\CO_R(H)$-module 
$
\Ext^j_{R,\Gb^-}
(\Delta_R(\lambda),\CO_R(B^-)_{\ad})
$
is finitely generated.
\end{proposition}
We first show the following.
\begin{lemma}
\label{lem:FG1}
Let $R$ be a commutative $\BA$-algebra.
For $M\in\Mod_{\inte}(U^{L}_R(\Gh))$ and $\xi\in P$ we have
\[
\Ind^{\Gb^-,\Gh}_{R}(M)_\xi
\cong
\bigoplus_{\mu\in P, \gamma\in Q^+,\mu+\gamma=\xi}
(\tilde{U}^{L}_{R}(\Gn^-)_{-\gamma})^*
\otimes
M_\mu.
\]
\end{lemma}
\begin{proof}
By definition we have
\[
\Ind^{\Gb^-,\Gh}_{R}(M)
=
\bigoplus_{\mu\in P}
\CO_R(B^-)^r_\mu\otimes
M_\mu,
\]
where 
\begin{align*}
\CO_R(B^-)^r_\mu
=&
\{\varphi\in \CO_R(B^-)\mid 
\varphi \cdotp h=\chi_\mu(h)\varphi
\;\;(h\in U^{L}_R(\Gh))\}.
\end{align*}
Under the identification 
\[
\CO_R(B^-)=\tilde{U}^{L}_{R}(\Gn^-)^\bigstar
\otimes \CO_R(H)
=\bigoplus_{\mu\in P}
\tilde{U}^{L}_{R}(\Gn^-)^\bigstar\chi_\mu
\]
of Proposition \ref{prop:BHN}
we have 
$
\CO_R(B^-)^r_\mu
=
\tilde{U}^{L}_{R}(\Gn^-)^\bigstar\chi_\mu,
$
and hence 
\[
\Ind^{\Gb^-,\Gh}_{R}(M)
=
\bigoplus_{\mu\in P}
\tilde{U}^{L}_{R}(\Gn^-)^\bigstar\chi_\mu\otimes
M_\mu.
\]
Then we see easily that 
\[
\Ind^{\Gb^-,\Gh}_{R}(M)_\xi
=
\bigoplus_{\mu+\gamma=\xi}
(\tilde{U}^{L}_{R}(\Gn^-)_{-\gamma})^*\chi_\mu\otimes
M_\mu
\cong
\bigoplus_{\mu+\gamma=\xi}
(\tilde{U}^{L}_{R}(\Gn^-)_{-\gamma})^*\otimes
M_\mu.
\]
\end{proof}
Consider the following two conditions on 
$M\in\Mod_{\inte}(U^{L}_R(\Gh))$:
\begin{itemize}
\item[(f1){$_M$}]
for any $\mu\in P$ the weight space $M_\mu$ is a finitely generated $R$-module;
\item[(f2){$_M$}]
for any $\mu\in P$ we have 
\[
\#\{
\delta\in Q^+\mid M_{\mu-\delta}\ne0\}<\infty.
\]
\end{itemize}
\begin{lemma}
\label{lem:FG2}
Let $R$ be a commutative $\BA$-algebra.
Assume 
$M\in\Mod_{\inte}(U^{L}_R(\Gh))$
satisfies $({\rm{f1}})_M$, 
$({\rm{f2}})_M$.
Then 
$N=
\Ind^{\Gb^-,\Gh}_{R}(M)
$
regarded as an object of $\Mod_{\inte}(U^{L}_R(\Gh))$
satisfies
$({\rm{f1}})_N$, 
$({\rm{f2}})_N$.
\end{lemma}
\begin{proof}
By Lemma \ref{lem:FG1} we have
\[
N_{\xi}
\cong
\bigoplus_{\gamma\in Q^+}
(\tilde{U}^{L}_{R}(\Gn^-)_{-\gamma})^*
\otimes
M_{\xi-\gamma},
\]
and hence$({\rm{f1}})_N$ holds.
Assume $N_{\xi-\delta}\ne0$ for $\xi\in P$ and $\delta\in Q^+$.
Then we have $M_{\xi-\delta-\gamma}\ne0$ for some $\gamma\in Q^+$.
Namely, there exists 
$\gamma'\in Q^+$ such that 
$M_{\xi-\gamma'}\ne0$ and $\gamma'-\delta\in Q^+$.
For each $\xi$ there exist only finitely many $\gamma'\in Q^+$ satisfying $M_{\xi-\gamma'}\ne0$.
For such $\gamma'$ there exist only finitely many $\delta\in Q^+$ satisfying $\gamma'-\delta\in Q^+$.
Hence $({\rm{f2}})_N$  holds.
\end{proof}
\begin{lemma}
\label{lem:FG3}
The conditions $({\rm{f1}})_M$ and $({\rm{f2}})_M$ hold for 
$M=\CO_R(B^-)_{\ad}$ 
regarded as an object of $\Mod_\inte(U^L_{\CO_R(H)}(\Gh))$.
\end{lemma}
\begin{proof}
This is a consequence of 
\[
(\CO_R(B^-)_{\ad})_\mu
=
(\tilde{U}^L_R(\Gn^-)_{-\mu})^*\CO_R(H)
\qquad(\mu\in P).
\]
\end{proof}

Now let us give a proof of Proposition \ref{prop:FG}.
Let $\CO_R(B^-)_\ad\to Q^\bullet$ be the standard resolution of $\CO_R(B^-)_\ad
\in\Mod_\inte(U^L_R(\Gb^-))$.
By 
\[
\Ext^j_{R,\Gb^-}
(\Delta_R(\lambda),\CO_R(B^-)_{\ad})
=H^j(
\Hom_{R,\Gb^-}
(\Delta_R(\lambda),Q^\bullet)
)
\]
it is sufficient to show that 
$\Hom_{R,\Gb^-}
(\Delta_R(\lambda),Q^j)$ is a finitely generated $\CO_R(H)$-module.
Since $\Delta_R(\lambda)$ is generated by the highest weight vector as a $U^L_R(\Gb^-)$-module, we have an embedding 
\[
\Hom_{R,\Gb^-}
(\Delta_R(\lambda),Q^j)
\hookrightarrow
Q^j_\lambda
\]
of $\CO_R(H)$-modules.
Hence it is sufficient to show that 
$Q^j_\lambda$ is a finitely generated $\CO_R(H)$-module.
We will show 
$({\rm{f1}})_{Q^j}$ together with $({\rm{f2}})_{Q^j}$ by induction on $j$.
Recall 
\[
Q^j=\Ind^{\Gb^-,\Gh}_{\CO_R(H)}(\tilde{Q}^j),
\qquad
\tilde{Q}^{j+1}=\Cok(\tilde{Q}^{j}\to Q^{j})
\]
for $j\geqq0$, where $\tilde{Q}^0=\CO_R(B^-)_\ad$.
By Lemma \ref{lem:FG3} the conditions 
$({\rm{f1}})_{\tilde{Q}^0}$, $({\rm{f2}})_{\tilde{Q}^0}$ are satisfied.
Hence $({\rm{f1}})_{{Q}^0}$, $({\rm{f2}})_{{Q}^0}$ hold by Lemma \ref{lem:FG2}.
Assume 
$({\rm{f1}})_{{Q}^j}$, $({\rm{f2}})_{{Q}^j}$ hold.
Since $\tilde{Q}^{j+1}$ is a quotient of $Q^j$, the conditions 
$({\rm{f1}})_{\tilde{Q}^{j+1}}$, $({\rm{f2}})_{\tilde{Q}^{j+1}}$ are satisfied.
Hence $({\rm{f1}})_{{Q}^{j+1}}$, $({\rm{f2}})_{{Q}^{j+1}}$ hold by Lemma \ref{lem:FG2}.

The proof of Proposition \ref{prop:FG} is now complete.

\section{Adjoint action and the center}
\subsection{}
We have the adjoint action of $U_\BF(\Gg;P^\vee)$ on $U_\BF(\Gg;P^\vee)$ given by
\[
\ad(u)(u')=
\sum_{(u)}u_{(0)}u'(Su_{(1)})
\qquad
(u, u'\in U_\BF(\Gg;P^\vee)).
\]
Assume $\Gamma$ and $\Gamma'$ are $\BZ$-lattices of $\Gh^*$ such that $Q\subset \Gamma, \Gamma'\subset P^\vee$.
Then we have
\[
\ad(U_\BF(\Gg;\Gamma))(U_\BF(\Gg;\Gamma'))
\subset
U_\BF(\Gg;\Gamma').
\]
Moreover, we have
\begin{align*}
\ad(U^L_\BA(\Gg;\Gamma))(U_\BA(\Gg;\Gamma'))
\subset&
U_\BA(\Gg;\Gamma')
\qquad
(\Gamma\supset Q^\vee),
\\
\ad(U^L_\BA(\Gg;\Gamma))(U^L_\BA(\Gg;\Gamma'))
\subset&
U^L_\BA(\Gg;\Gamma')
\qquad
(\Gamma,\Gamma'\supset Q^\vee).
\end{align*}
In particular, for a commutative $\BA$-algebra $R$ 
we have the adjoint actions
\begin{align}
\label{eq:ad1}
&U^L_R(\Gg)\times U^L_R(\Gg)
\to
U^L_R(\Gg)
\qquad((u,u')\mapsto \ad(u)(u')),
\\
\label{eq:ad2}
&U^L_R(\Gg)\times U_R(\Gg)
\to
U_R(\Gg)
\qquad((u,u')\mapsto \ad(u)(u')).
\end{align}
Note that the adjoint action \eqref{eq:adO} of $U^L_R(\Gg)$ on $\CO_R(G)$  is related to \eqref{eq:ad1} by 
\begin{equation}
\label{eq:tad}
\langle\ad(u)(\varphi),u'\rangle
=
\langle\varphi,\ad(S^{-1}u)(u')\rangle
\qquad
(u, u'\in U^L_R(\Gg),\;\varphi\in \CO_R(G)).
\end{equation}
\subsection{}
For a commutative $\BA$-algebra $R$ 
we set
\[
{}^eU_R(\Gh)=\bigoplus_{\lambda\in P}
Rk_{2\lambda}
\subset U_R(\Gh), 
\]
and denote by
${}^eU_R(\Gg)$ the subalgebra of $U_R(\Gg)$ generated by ${}^eU_R(\Gh)$, $U_R(\Gn^+)$, $\tilde{U}_R(\Gn^-)$.
Then the multiplication of ${}^eU_R(\Gg)$ induces the isomorphism 
\[
{}^eU_R(\Gg)
\cong
\tilde{U}_R(\Gn^-)\otimes
{}^eU_R(\Gh)\otimes
U_R(\Gn^+)
\]
or $R$-modules.
We have
\[
\ad(U^L_R(\Gg))({}^eU_R(\Gg))
\subset
{}^eU_R(\Gg).
\]
\subsection{}
We set 
\[
{}^fU_\BF(\Gg)
=\{
u\in U_\BF(\Gg)\mid
\dim\ad(U^L_\BF(\Gg))(u)<\infty\}.
\]
It is a subalgebra of $U_\BF(\Gg)$.
By \cite{JL} we have
\[
{}^fU_\BF(\Gg)
=\bigoplus_{\lambda\in P^+}
\ad(U^L_\BF(\Gg))(k_{-2\lambda}).
\]
In particular, we have
$
{}^fU_\BF(\Gg)\subset {}^eU_\BF(\Gg).
$
Set ${}^fU_\BA(\Gg)={}^fU_\BF(\Gg)\cap U_\BA(\Gg)$.
For a commutative $\BA$-algebra $R$ 
we set
\[
{}^fU_R(\Gg)=\Image(
R\otimes_\BA{}^fU_\BA(\Gg)
\to U_R(\Gg))
\subset
{}^eU_R(\Gg).
\]
\subsection{}
Take a positive integer $N$ such that $(P^\vee, P^\vee)\subset\frac1N\BZ$ with respect to
\eqref{eq:bilin}.
We denote by 
\[
\tau:U_\BF(\Gb^+; P^\vee)\times U_\BF(\Gb^-; P^\vee)\to\BQ(q^{\pm{1/N}})
\]
the Drinfeld pairing.
Namely, $\tau$ is the unique $\BF$-bilinear map satisfying 
\begin{align}
&\tau(x,y_1y_2)
=(\tau\otimes\tau)(\Delta(x),y_1\otimes y_2)
\qquad
&(x\in U_\BF(\Gb^+),\;
y_1, y_2\in U_\BF(\Gb^-)),
\\
&\tau(x_1x_2,y)
=(\tau\otimes\tau)(x_2\otimes x_1,\Delta(y))
\qquad
&(x_1, x_2\in U_\BF(\Gb^+),\;
y\in U_\BF(\Gb^-)),
\\
&\tau(k_\lambda, k_\mu)=q^{-(\lambda,\mu)}
&(\lambda, \mu\in P^\vee),
\\
&\tau(k_\lambda, f_i)=\tau(e_i,k_\lambda)=0
&(\lambda\in P^\vee, \; i\in I),
\\
&\tau(e_i,f_j)=-\delta_{ij}(q_i-q_i^{-1})^{-1}
&(i, j\in I).
\end{align}
It restricts to $\BA$-bilinear maps
\begin{align}
\tau^L_\BA:
U_\BA(\Gb^+)\times U^L_\BA(\Gb^-)\to\BA,
\\
{}^L\tau_\BA:
U^L_\BA(\Gb^+)\times U_\BA(\Gb^-)\to\BA.
\end{align}
For a commutative $\BA$-algebra $R$ we denote their base changes by
\begin{align}
\tau^L_R:
U_R(\Gb^+)\times U^L_R(\Gb^-)\to R,
\\
{}^L\tau_R:
U^L_R(\Gb^+)\times U_R(\Gb^-)\to R.
\end{align}

\subsection{}
Let $R$ be a commutative $\BA$-algebra.
We define a bilinear form
\begin{equation}
\kappa_R:U^L_R(\Gg)\times {}^eU_R(\Gg)\to R
\end{equation}
by
\begin{align*}
&\kappa_R(yhx,y'k_{2\lambda}x')
=
\tau^L_R(x,y'){}^L\tau_R(x',S^2y)\chi_{-\lambda}(h)
\\
&\qquad
(x\in {U}^{L}_R(\Gn^+),\;
h\in {U}^{L}_R(\Gh),\;
y\in \tilde{U}^{L}_R(\Gn^-), \;
x'\in U_R(\Gn^+), y'\in \tilde{U}_R(\Gn^-),
\lambda\in P).
\end{align*}
It satisfies
\begin{equation}
\label{eq:tad2}
\kappa_R(\ad(z)(v),u)=\kappa_R(v,\ad(Sz)(u))
\qquad(z, v\in U^L_R(\Gg), \; u\in{}^eU_R(\Gg))
\end{equation}
(see \cite{T92}).
Define
\begin{equation}
\jmath_R: {}^eU_R(\Gg)\to
(U^L_R(\Gg))^*
\end{equation}
by
\[
\langle\jmath_R(u),v\rangle=\kappa_R(v,u)
\qquad
(u\in{}^eU_R(\Gg), \; v\in U^L_R(\Gg)).
\]
It follows from the explicit calculation of $\tau$ in terms of the PBW-type basis (\cite[38.2]{Lbook})  that $\jmath_R$ is injective.
Moreover, it preserves the adjoint action of $U^L_R(\Gg)$
by \eqref{eq:tad} and \eqref{eq:tad2}.
Namely, we have
\[
\jmath_R(\ad(u)(u'))=\ad(u)(\jmath_R(u'))
\qquad
(u\in U^L_R(\Gg),\; u'\in {}^eU_R(\Gg)).
\]
By Caldero \cite{Cal} $\jmath_\BF$ induces an isomorphism
$
{}^fU_\BF(\Gg)\cong\CO_\BF(G)_\ad
$
of $U^L_\BF(\Gg)$-modules.
Hence by $\CO_\BA(G)=\CO_\BF(G)\cap(U^L_\BA(\Gg))^*$ we obtain 
$
{}^fU_\BA(\Gg)\cong\CO_\BA(G)_\ad
$.
\begin{lemma}
For a commutative $\BA$-algebra $R$ we have
\begin{align}
\label{eq:Cal1}
&{}^fU_R(\Gg)
\cong
R\otimes_\BA{}^fU_\BA(\Gg),
\\
\label{eq:Cal2}
&{}^fU_R(\Gg)
\cong\CO_R(G)_\ad.
\end{align}
\end{lemma}
\begin{proof}
We have
\[
R\otimes_\BA{}^fU_\BA(\Gg)
\cong
R\otimes_\BA\CO_\BA(G)_\ad
\cong
\CO_R(G)_\ad
\]
by Proposition \ref{prop:bcO}.
Hence the desired results follow from the following commutative diagram:
\[
\xymatrix{
R\otimes_\BA{}^fU_\BA(\Gg)
\ar@{->>}[d]
\ar[r]^{\;\;\sim}
&
\CO_R(G)_\ad
\ar@{^{(}-_>}[d]
\\
{}^fU_R(\Gg)
\ar@{^{(}-_>}[r]^{\jmath_R}
&
(U^L_R(\Gg))^*.
}
\]
\end{proof}

\subsection{}
For a commutative ring $R$ we denote by $R[P]$ the group algebra of $P$.
Namely, $R[P]$ is a free $R$-module with basis $\{e(\lambda)\mid\lambda\in P\}$, and its multiplication  is given by
\[
e(\lambda) e(\mu)=e(\lambda+\mu)
\qquad(\lambda, \mu\in P).
\]
In the case $R$ is a commutative $\BA$-algebra
we define an $R$-linear $W$-action 
\[
W\times R[P]\to R[P]
\qquad
((w,f)\mapsto w\circ f)
\]
on $R[P]$ by
\begin{equation}
w\circ e(\lambda)=q^{(w\lambda-\lambda,\rho)}e(w\lambda)
\qquad
(w\in W, \lambda\in P),
\end{equation}
where 
\[
\rho=\frac12\sum_{\alpha\in\Delta^+}\alpha\in P.
\]

\subsection{}
We see easily that 
\[
Z(U_\BF(\Gg))
=
U_\BF(\Gg)^{\ad(U^L_\BF(\Gg))}
=:
\{
z\in U_\BF(\Gg)
\mid
\ad(u)(z)=\varepsilon(u)z\;(u\in U^L_\BF(\Gg))\}
\subset
{}^fU_\BF(\Gg).
\]
Hence by $Z(U_\BA(\Gg))=Z(U_\BF(\Gg))\cap U_\BA(\Gg)$ we have also
\begin{equation}
\label{eq:ad-inv}
Z(U_\BA(\Gg))
=
U_\BA(\Gg)^{\ad(U_\BA^L(\Gg))}
\subset
{}^fU_\BA(\Gg).
\end{equation}

Define
\begin{equation}
\iota:Z(U_\BF(\Gg))\to\BF[P]
\end{equation}
as the composite of 
\[
Z(U_\BF(\Gg))
\subset
U_\BF(\Gg)
\cong
\tilde{U}_\BF(\Gn^-)\otimes U_\BF(\Gh)\otimes
U_\BF(\Gn^+)
\xrightarrow{\varepsilon\otimes1\otimes\varepsilon}
U_\BF(\Gh)
\cong\BF[P].
\]
Here, the isomorphism
${U}_\BF(\Gg)
\cong
\tilde{U}_\BF(\Gn^-)\otimes U_\BF(\Gh)\otimes
U_\BF(\Gn^+)$ is given by the multuplication of $U_\BF(\Gg)$, and 
the identification 
$U_\BF(\Gh)
\cong\BF[P]$ is given by
$
k_\lambda\leftrightarrow e(\lambda)$.

\begin{proposition}
\label{prop:HC}
\begin{itemize}
\item[(i)]
The linear map $\iota:Z(U_\BF(\Gg))\to\BF[P]$ is an injective algebra homomorphism. 
Its image coincides with 
\[
\BF[2P]^{W\circ}
=
\{f\in\BF[2P]\mid w\circ f=f\;\;(w\in W)\}.
\]
In particular, we have $Z(U_\BF(\Gg))\cong\BF[2P]^{W\circ}$.
\item[(ii)]
The restriction of $\iota$ to $Z(U_\BA(\Gg))$ gives the isomorphism 
$Z(U_\BA(\Gg))\cong\BA[2P]^{W\circ}$.
\end{itemize}
\end{proposition}
\begin{proof}
(i) is well-known (see for example \cite{T92}).
From this we obtain  $Z(U_\BA(\Gg))\hookrightarrow\BA[2P]^{W\circ}$.
Its surjectivity is proved similarly to the proof of the surjectivity of $Z(U_\BF(\Gg))\to\BF[2P]^{W\circ}$ given in \cite{T92}.
\end{proof}
For a commutative $\BA$-algebra $R$
we set
\[
Z_{\Har}(U_R(\Gg))=\Image(R\otimes_\BA Z(U_\BA(\Gg))\to U_R(\Gg)).
\]
It is a central subalgebra of $U_R(\Gg)$.

\subsection{}

Although the isomorphism \eqref{eq:Cal2} does not respect the multiplication, we have the following.

\begin{lemma}
\label{lem:jmath}
For $u\in {}^eU_R(\Gg)$ and $z\in Z_\Har(U_R(\Gg))$
we have
\[
\jmath_R(uz)=\jmath_R(u)\jmath_R(z).
\]
\end{lemma}
\begin{proof}
We may assume $R=\BA$.
We can write
\begin{align*}
z=\sum_iy_ik_{2\lambda_i}x_i
\qquad
(y_i\in \tilde{U}_\BA(\Gn^-)_{-\gamma_i},\;
\lambda_i\in P,\;
x_i\in U_\BA(\Gn^+)_{\gamma_i},\; 
\gamma_i\in Q^+).
\end{align*}
We may assume 
\begin{align*}
u=yk_{2\mu}x
\qquad
(y\in \tilde{U}_\BA(\Gn^-)_{-\delta},\;
\mu\in P,\;
x\in U_\BA(\Gn^+), \;
\delta\in Q^+).
\end{align*}
Then for 
$\tilde{y}\in \tilde{U}^{L}_\BA(\Gn^-),\; 
t\in U_\BA^{L}(\Gh), \; 
\tilde{x}\in U_\BA^{L}(\Gn^+)$
we have
\begin{align*}
\langle
\jmath_\BA(u)\jmath_\BA(z),
\tilde{y}t\tilde{x}
\rangle
=&
\sum_{(\tilde{y}), (t), (\tilde{x})}
\langle
\jmath_\BA(u),
\tilde{y}_{(0)}t_{(0)}\tilde{x}_{(0)}
\rangle
\langle
\jmath_\BA(z),
\tilde{y}_{(1)}t_{(1)}\tilde{x}_{(1)}
\rangle
\\
=&
\sum_{i, (\tilde{y}), (t), (\tilde{x})}
\kappa_\BA(
\tilde{y}_{(0)}t_{(0)}\tilde{x}_{(0)},
yk_{2\mu}x
)
\kappa_\BA(
\tilde{y}_{(1)}t_{(1)}\tilde{x}_{(1)},
y_ik_{2\lambda_i}x_i
).
\end{align*}
Note
\begin{align*}
\Delta(U_\BA^{L}(\Gn^+))
\subset&
\sum_{\gamma\in Q^+}
U_\BA^{L}(\Gn^+)k_\gamma\otimes U_\BA^{L}(\Gn^+)_\gamma,
\\
\Delta(\tilde{U}_\BA^{L}(\Gn^-))
\subset&
\sum_{\gamma\in Q^+}
\tilde{U}_\BA^{L}(\Gn^-)k_\gamma\otimes 
\tilde{U}_\BA^{L}(\Gn^-)_{-\gamma}.
\end{align*}
By the definition of $\kappa_\BA$ 
we only need to consider the terms satisfying
\[
\tilde{y}_{(1)}\in 
\tilde{U}_\BA^{L}(\Gn^-)_{-\gamma_i},
\qquad
\tilde{x}_{(1)}\in
U_\BA^{L}(\Gn^+)_{\gamma_i}.
\]
In this case we have
\[
\tilde{y}_{(0)}\in 
\tilde{U}_\BA^{L}(\Gn^-)k_{\gamma_i},
\qquad
\tilde{x}_{(0)}\in
k_{\gamma_i}U_\BA^{L}(\Gn^+).
\]
By
\[
\tilde{y}_{(0)}t_{(0)}\tilde{x}_{(0)}
=
(\tilde{y}_{(0)}k_{-\gamma_i})
(t_{(0)}k_{2\gamma_i})
(k_{-\gamma_i}\tilde{x}_{(0)}).
\]
we have
\begin{align*}
&\langle
\jmath_\BA(u)\jmath_\BA(z),
\tilde{y}t\tilde{x}
\rangle
\\
=&
\sum_{i, (\tilde{y}), (t), (\tilde{x})}
\kappa_\BA(
(\tilde{y}_{(0)}k_{-\gamma_i})
(t_{(0)}k_{2\gamma_i})
(k_{-\gamma_i}\tilde{x}_{(0)}),
yk_{2\mu}x
)
\times
\kappa_\BA(
\tilde{y}_{(1)}t_{(1)}\tilde{x}_{(1)},
y_ik_{2\lambda_i}x_i
)
\\
=&
\sum_{i, (\tilde{y}), (t), (\tilde{x})}
\left(
{}^L\tau_\BA(k_{-\gamma_i}\tilde{x}_{(0)}, y)
\tau^L_\BA(x,S^2(\tilde{y}_{(0)}k_{-\gamma_i}))
\chi_{-\mu}(t_{(0)}k_{2\gamma_i})
\right)
\\
&\qquad\qquad\qquad
\times
\left(
{}^L\tau_\BA(\tilde{x}_{(1)},y_i)
\tau^L_\BA(x_i,S^2\tilde{y}_{(1)})
\chi_{-\lambda_i}(t_{(1)})
\right).
\end{align*}
By $y\in \tilde{U}_\BA^{L}(\Gn^-)_{-\delta}={U}_\BA^{L}(\Gn^-)_{-\delta}k_{\delta}$ 
we only need to consider the terms satisfying
$\tilde{x}_{(0)}\in {U}_\BA^{L}(\Gn^+)_\delta k_{\gamma_i}$.
Then we have
\begin{align*}
{}^L\tau_\BA(\tilde{x}_{(0)}, y)
=&
{}^L\tau_\BA((\tilde{x}_{(0)}k_{-\gamma_i})k_{\gamma_i}, 
(yk_{-\delta})k_\delta
)
=
q^{-(\gamma_i,\delta)}
{}^L\tau_\BA(\tilde{x}_{(0)}k_{-\gamma_i}, 
yk_{-\delta})
\\
=&
q^{-(\gamma_i,\delta)}
{}^L\tau_\BA(\tilde{x}_{(0)}k_{-\gamma_i}, 
y)
=
{}^L\tau_\BA(k_{-\gamma_i}\tilde{x}_{(0)}, y).
\end{align*}
Hence we have
\begin{align*}
&\langle
\jmath_\BA(u)\jmath_\BA(z),
\tilde{y}t\tilde{x}
\rangle
\\
=&
\sum_{i, (\tilde{y}), (\tilde{x})}
{}^L\tau_\BA(\tilde{x}_{(0)}, y)
\tau_\BA^L(x,S^2\tilde{y}_{(0)})
{}^L\tau_\BA(\tilde{x}_{(1)},y_i)
\tau_\BA^L(x_i,S^2\tilde{y}_{(1)})
\chi_{-(\mu+\lambda_i)}(t)
q^{-2(\mu,\gamma_i)}
\\
=&
\sum_{i}
{}^L\tau_\BA(\tilde{x},yy_i)
\tau^L_\BA(x_ix,S^2\tilde{y})
\chi_{-(\mu+\lambda_i)}(t)
q^{-2(\mu,\gamma_i)}
\\
=&
\sum_i
q^{-2(\mu,\gamma_i)}
\kappa_\BA(\tilde{y}t\tilde{x},(yy_i)k_{2(\mu+\lambda_i)}(x_ix))
\\
=&
\sum_i
\kappa_\BA(\tilde{y}t\tilde{x},yy_ik_{2\lambda_i}x_ik_{2\mu}x)
\\
=&
\kappa_\BA(\tilde{y}t\tilde{x},yzk_{2\mu}x)
=
\kappa_\BA(\tilde{y}t\tilde{x},yk_{2\mu}xz)
=
\kappa_\BA(\tilde{y}t\tilde{x},uz)
=
\langle
\jmath_\BA(uz),
\tilde{y}t\tilde{x}
\rangle.
\end{align*}
\end{proof}
\subsection{}
In general for a commutative $\BA$-algebra $R$, 
we define an $R$-linear $W$-action 
\begin{equation}
\label{eq:W-bul1}
W\times \CO_R(H)\to \CO_R(H)
\qquad
((w,\chi)\mapsto w\bullet\chi)
\end{equation}
by
\[
w\bullet\chi_\lambda=
q^{-2(w\lambda-\lambda,\rho)}\chi_{w\lambda}
\qquad
(w\in W, \lambda\in P).
\]
This induces a $W$-action
\begin{equation}
\label{eq:W-bul2}
W\times H(R)\to H(R)
\qquad
((w,t)\mapsto w\bullet t)
\end{equation}
on $H(R)$ given by 
\[
w\bullet t=w(tt_{-2\rho})t_{2\rho}
\qquad
(w\in W, \; t\in H(R)),
\]
where $t_{\pm2\rho}\in H(R)$ is defined by $\langle\chi_\lambda,t_{\pm2\rho}\rangle
=
q^{\pm2(\lambda,\rho)}$ for any $\lambda\in P$.

Then identifying $R[2P]$ with $\CO_R(H)$ via
$
e(2\lambda)\leftrightarrow\chi_{-\lambda}$ ($\lambda\in P$) 
we have 
\begin{equation}
\label{eq:bw}
R[2P]^{W\circ}\cong \CO_R(H)^{W\bullet}.
\end{equation}

By Lemma \ref{lem:jmath}
$\CO_\BA(G)$ contains a subalgebra 
$\jmath_\BA(Z(U_\BA(\Gg)))$, which is isomorphic to
$Z(U_\BA(\Gg))$.
By \eqref{eq:ad-inv}
we have
\[
\jmath_\BA(Z(U_\BA(\Gg)))
=\CO_\BA(G)^{\ad(U_\BA^L(\Gg))}
:=
\{\varphi\in\CO_\BA(G)\mid
\ad(u)(\varphi)=\varepsilon(u)\varphi\;(u\in U_\BA^L(\Gg))\}
.
\]
By the definition of $\jmath_\BA$ we have the following commutative diagram:
\[
\xymatrix{
Z(U_\BA(\Gg))
\ar[r]^{\jmath_\BA\quad}
\ar[d]_{\iota}
&
\CO_\BA(G)^{\ad(U_\BA^L(\Gg))}
\ar[d]^r
\\
\BA[2P]
\ar[r]_f
&
\CO_\BA(H).
}
\]
Here, $r$ is the restriction of the canonical Hopf algebra homomorphism $\CO_\BA(G)\to\CO_\BA(H)$ corresponding to $U^L_\BA(\Gh)\subset U^L_\BA(\Gg)$, and $f$ is given by $e(2\lambda)\mapsto \chi_{-\lambda}$.
Note that $\jmath_\BA$ and $f$ are isomorphisms.
Moreover $\iota$ induces $Z(U_\BA(\Gg))\cong\BA[2P]^{W\circ}$.
Hence $r$ induces the isomorphism 
\[
\CO_\BA(G)^{\ad(U_\BA^L(\Gg))}
\cong
\CO_\BA(H)^{W\bullet}.
\]

\begin{proposition}
\label{prop:HC2}
For a commutative $\BA$-algebra $R$ we have
\[
R\otimes_\BA
\CO_\BA(G)^{\ad(U_\BA^L(\Gg))}
\cong
\CO_R(G)^{\ad(U_R^L(\Gg))}
\cong
\CO_R(H)^{W\bullet}.
\]
Here, the second isomorphism is obtained by restricting the canonical Hopf algebra homomorphism $\CO_R(G)\to\CO_R(H)$.
\end{proposition}
\begin{proof}
We have
\[
R\otimes_\BA\CO_\BA(G)^{\ad(U_\BA^L(\Gg))}
\cong
R\otimes_\BA\CO_\BA(H)^{W\bullet}
\cong
\CO_R(H)^{W\bullet}.
\]
The first isomorphism is a consequence of Proposition \ref{prop:dualWeyl2} (ii) applied to $\lambda=0$.
\end{proof}
\begin{corollary}
\label{cor:HC2}
We have
\[
R\otimes_\BA Z(U_\BA(\Gg))
\cong
Z_{\Har}(U_R(\Gg))
=
{}^eU_R(\Gg)^{\ad(U^L_R(\Gg))}\
\cong
R[2P]^{W\circ}
.
\]
\end{corollary}
By \eqref{eq:bw} we have a natural isomorphism
\begin{equation}
\label{eq:bw2}
Z_\Har(U_R(\Gg))\cong \CO_R(H)^{W\bullet}
\end{equation}
of $R$-algebras.
For $t\in H(R)$ we define an algebra homomorphism
\begin{equation}
\label{eq:HC-ch}
\xi_{\Har,t}:Z_\Har(U_R(\Gg))\to R
\end{equation}
as the composition of \eqref{eq:bw2} and $
\CO_R(H)^{W\bullet}\hookrightarrow\CO_R(H)\xrightarrow{t} R$.

\subsection{}

We denote the composition 
\[
\CO_R(H)^{W\bullet}\
\cong
\CO_R(G)^{\ad(U_R^L(\Gg))}
\hookrightarrow
\CO_R(G)
\]
by
\begin{equation}
\vartheta:\CO_R(H)^{W\bullet}\to\CO_R(G).
\end{equation}
It is an embedding of $R$-algebras.

For $\lambda\in P^+$ we define 
$c(\lambda)\in\CO_R(H)^{W\bullet}$ by
\[
c(\lambda)
=
\sum_{\mu\in P}q^{-2(\rho,\mu)}\rank \,\Delta_R(\lambda)_\mu \,\chi_{\mu}.
\]
Note that $\{c(\lambda)\mid\lambda\in P^+\}$ is an $R$-basis of $\CO_R(H)^{W\bullet}$.
\begin{lemma}
\label{lem:theta}
For $\lambda\in P^+$ we have
\[
\langle\vartheta(c(\lambda)),u\rangle
=\Trace(uk_{-2\rho}:\Delta_R(\lambda)\to \Delta_R(\lambda))
\qquad(u\in U^L_\BA(\Gg)).
\]
\end{lemma}
\begin{proof}
Define $\varphi\in\CO_R(G)$ by
\[
\langle\varphi,u\rangle
=\Trace(uk_{-2\rho}:\Delta_R(\lambda)\to \Delta_R(\lambda))
\qquad(u\in U^L_\BA(\Gg)).
\]
We can easily check that $\varphi\in\CO_R(G)^{\ad(U_R^L(\Gg))}$ and $\varphi|_{U_R^L(\Gh)}=c(\lambda)$.
\end{proof}

We will regard $\CO_R(G)$ as a right $\CO_R(H)^{W\bullet}$-module 
by
\[
\CO_R(G)\times \CO_R(H)^{W\bullet}
\to
\CO_R(H)^{W\bullet}
\qquad
((\varphi,f)\mapsto\varphi\,\vartheta(f)).
\]
The following fact is proved similarly to Lemma \ref{lem:seB}.
\begin{lemma}
\label{lem:seG}
The adjoint action of $U_R(\Gg)$ on $\CO_R(G)$ is 
$\CO_R(H)^{W\bullet}$-linear.
\end{lemma}
Hence we can regard $\CO_R(G)_\ad$ as an object of
$\Mod_\inte(U^L_{\CO_R(H)^{W\bullet}}(\Gg))$.

\subsection{}
Now we  construct a homomorphism
\begin{equation}
\label{eq:F}
F:\CO_R(G)_\ad\otimes_{\CO_R(H)^{W\bullet}}\CO_R(H)
\to
\Ind^{\Gg,\Gb^-}_{R}(\CO_R(B^-)_\ad)
\end{equation}
in $\Mod_\inte(U^L_{\CO_R(H)}(\Gg))$.
Note that by
\[
\Ind^{\Gg,\Gb^-}_{R}(\CO_R(B^-)_\ad)
\cong
\Ind^{\Gg,\Gb^-}_{\CO_R(H)}(\CO_R(B^-)_\ad)
\]
we have
\[
\Ind^{\Gg,\Gb^-}_{R}(\CO_R(B^-)_\ad)
\in \Mod_\inte(U^L_{\CO_R(H)}(\Gg)).
\]

By the Frobenius reciprocity we have
\[
\Hom_{R,\Gg}(\CO_R(G)_\ad,
\Ind^{\Gg,\Gb^-}_{R}(\CO_R(B^-)_\ad))
\cong
\Hom_{R,\Gb^-}(\CO_R(G)_\ad,
\CO_R(B^-)_\ad).
\]
Define 
\[
F'\in \Hom_{R,\Gg}(\CO_R(G)_\ad,
\Ind^{\Gg,\Gb^-}_{R}(\CO_R(B^-)_\ad))
\]
to be the homomorphism corresponding to the canonical Hopf algebra homomorphism $\res:\CO_R(G)\to\CO_R(B^-)$ associated to the embedding 
$U^L_R(\Gb^-)\subset U^L_R(\Gg)$.
Let us show that $F'$ preserves the $\CO_R(H)^{W\bullet}$-module structure.
It is sufficient to show that 
$\res$ preserves  the $\CO_R(H)^{W\bullet}$-module structure.
Hence we have only to show 
$(\res\circ\vartheta)(f)=f$ for any $f\in \CO_R(H)^{W\bullet}$.
This follows easily from Lemma \ref{lem:theta}.
Since $F'$ is a homomorphism of $\CO_R(H)^{W\bullet}$-modules, we can define \eqref{eq:F}
by
\[
F(\varphi\otimes\chi)=F'(\varphi)\chi
\qquad(\varphi\in\CO_R(G)_\ad, \; \chi\in\CO_R(H)).
\]
\begin{remark}
For $\varphi\in\CO_R(G)_\ad$ we have
\[
F(\varphi)
=
\sum_{(\varphi)}
\varphi_{(2)}(S^{-1}\varphi_{(0)})
\otimes
\res(\varphi_{(1)})
\in
\Ind^{\Gg,\Gb^-}_{R}(\CO_R(B^-)_\ad)
\subset
\CO_R(G)\otimes\CO_R(B^-)_\ad.
\]
\end{remark}
\begin{conjecture}
\label{conj:IV}
$F$ is an isomorphism.
Moreover, 
we have
\[
R^j\Ind^{\Gg,\Gb^-}_{R}(\CO_R(B^-)_\ad)=0
\qquad(j>0).
\]
\end{conjecture}

Let $t\in H(R)$, and set
\begin{align*}
\CO_{R}(G)_{\ad,t}=&
\CO_{R}(G)_{\ad}\otimes_{\CO_R(H)^{W\bullet}}
R\in\Mod_\inte(U^L_R(\Gg)),
\\
\CO_{R}(B^-)_{\ad,t}=&
\CO_{R}(B^-)_{\ad}\otimes_{\CO_R(H)}
R
\in\Mod_\inte(U^L_R(\Gb^-)),
\end{align*}
where the module structure on $R$ in the formulas above 
is defined by the homomorphisms $\CO_R(H)^{W\bullet}
\hookrightarrow
\CO_R(H)
\xrightarrow{t}R$ and 
$\CO_R(H)
\xrightarrow{t}R$, respectively.

Since $\res:\CO_R(G)\to\CO_R(B^-)$ is a homomorphism of $\CO_R(H)^{W\bullet}$-modules, we have
\[
\res\in
\Hom_{\CO_R(H)^{W\bullet},\Gb^-}(\CO_R(G)_\ad,
\CO_R(B^-)_\ad).
\]
Applying $(\bullet)\otimes_{\CO_R(H)^{W\bullet}}R$
we obtain 
\[
r_t=\res\otimes1
\in
\Hom_{R,\Gb^-}(\CO_R(G)_{\ad,t},
\CO_R(B^-)_{\ad,t}).
\]
We denote by
\[
F_t\in
\Hom_{R,\Gg}(\CO_R(G)_{\ad,t},
\Ind^{\Gg,\Gb^-}_{R}(\CO_R(B^-)_{\ad,t}))
\]
the homomorphism corresponding to $r_t$ under the Frobenius reciprocity.

\begin{conjecture}
\label{conj:IVt}
$F_t$ is an isomorphism.
Moreover, 
we have
\[
R^j\Ind^{\Gg,\Gb^-}_{R}(\CO_R(B^-)_{\ad,t})=0
\qquad(j>0).
\]
\end{conjecture}

\subsection{}
We set 
\begin{equation}
V_R=
{}^eU_R(\Gg)/
(\Ker(\varepsilon)\cap \tilde{U}_R(\Gn^-))
{}^eU_R(\Gg).
\end{equation}
The adjoint action of $U^L_R(\Gb^-)$ on ${}^eU_R(\Gg)$ induces  a left $U^L_R(\Gb^-)$-module structure of $V_R$.
By the definition of $\kappa_R$ we see easily that 
$\jmath_R$ induces an injective $R$-homomorphism
\begin{equation}
\overline{\jmath}_R:V_R\to U^L_R(\Gb^-)^*.
\end{equation}
Moreover, it is easily seen that its image coincides with $\CO_R(B^-)$.
Hence we obtain an isomorphism
\begin{equation}
\label{eq:VO}
V_R\cong\CO_R(B^-)_\ad
\end{equation}
in $\Mod_\inte(U^L_R(\Gb^-))$.
Regard $V_R$ as an $\CO_R(H)$-module by
\[
\overline{u}\chi_\lambda
=\overline{uk_{-2\lambda}}
\qquad
(u\in {}^eU_R(\Gg),\; \lambda\in P).
\]
Then \eqref{eq:VO} turns out to be an isomorphism in 
$\Mod_\inte(U^L_{\CO_R(H)}(\Gb^-))$.

For $t\in H(R)$ we set
\begin{equation}
V_{R,t}=V_R\otimes_{\CO_R(H)}R
\in
\Mod_\inte(U^L_{R}(\Gb^-)),
\end{equation}
where the module structure on $R$ is defined by the homomorphism $t:\CO_R(H)\to R$.
By \eqref{eq:VO} we have
\begin{equation}
\label{eq:VOt}
V_{R,t}\cong\CO_R(B^-)_{\ad,t}.
\end{equation}

\section{Review of \cite{T2}}
\label{sec:review}
\subsection{}
When we regard $\BC$ as an $\BA$-algebra via $q\mapsto\zeta\in\BC^\times$, we denote it as $\BC_\zeta$.
For simplicity we set
$
U_\zeta=U_{\BC_\zeta}(\Gg).
$
Moreover, for $t\in H=H(\BC)$ we set
$
U_{\zeta,t}
=
U_\zeta\otimes_{Z_\Har(U_\zeta)}\BC
$ with respect to the specialization  $\xi_{\Har,t}:Z_\Har(U_\zeta)\to\BC$
(see \eqref{eq:HC-ch}).
\subsection{}
For $\zeta\in\BC^\times$
we can construct a non-commutative projective scheme 
$\CB_\zeta$ over $\BC$, which is called the quantized flag manifold, 
and a sheaf $\DD_{\CB_\zeta}$ of rings on it.
We do not recall the definition here (see \cite{T1}, \cite{T2}).
The sheaf $\DD_{\CB_\zeta}$ contains $\CO_\BC(H)$ as a central subring, 
and for $t\in H$ we can consider the specialization
\[
\DD_{\CB_\zeta, t}:=
\DD_{\CB_\zeta}\otimes_{\CO_\BC(H)}\BC,
\]
where the action of $\CO_\BC(H)$ on $\BC$ is defined by the homomorphism $t:\CO_\BC(H)\to\BC$.
Note that $\CB_\zeta$ is a non-commutative  analogue of the ordinary flag manifold $\CB=B^-\backslash G$, and $\DD_{\CB_\zeta, t}$
is an analogue of the sheaf of twisted differential operators on $\CB$.
We have natural algebra homomorphisms
\begin{align}
\label{eq:canon0}
U_{\zeta}\otimes_{Z_{\Har}(U_\zeta)}\CO_\BC(H)
\to
\Gamma(\CB_\zeta,\DD_{\CB_\zeta}),
\\
\label{eq:canont}
U_{\zeta,t}
\to
\Gamma(\CB_\zeta,\DD_{\CB_\zeta,t})
\qquad(t\in H).
\end{align}
\begin{conjecture}
\label{conj:D}
For $\zeta\in\BC^\times$ and $t\in H$
the algebra homomorphism
\eqref{eq:canont} 
is an isomorphism and we have
\[
H^j(\CB_\zeta,\DD_{\CB_\zeta,t})=0
\qquad(j>0).
\]
\end{conjecture}
Denote the category of coherent $\DD_{\CB_\zeta,t}$-modules by $\Mod_\coh(\DD_{\CB_\zeta,t})$ and that of finitely generated $U_{\zeta,t}$-modules by $\Mod_f(U_{\zeta,t})$.

When $\zeta$ is transcendental and $t$ is dominant integral, we proved Conjecture \ref{conj:D} and established the 
Beilinson-Bernstein type equivalence 
\[
\Mod_{\coh}(\DD_{\CB_\zeta,t})
\cong
\Mod_{f}(U_{\zeta,t})
\]
of abelian categories in \cite{T0}.

Now consider the case when $\zeta$ is a root of 1.
For a positive integer $\ell$ we set
\[
\zeta_\ell=\exp\left(\frac{2\pi\sqrt{-1}}{\ell}\right)\in\BC^\times.
\]
We assume that the integer $\ell>1$ satisfies
\begin{itemize}
\item[(a1)]
$\ell$ is odd;
\item[(a2)]
$\ell$ is prime to  $|P/Q|$;
\item[(a3)]
$\ell$ is prime to  $3$ if $\Delta$ is of type $G_2$
\end{itemize}
in the following.
We say that $t\in H$ is regular if $|W\bullet t|=|W|$ with respect to \eqref{eq:W-bul2}.

\begin{proposition}[\cite{T2}]
\label{prop:D-eq}
For $\ell>1$ satisfying 
$({\rm{a1}})$, $({\rm{a2}})$, $({\rm{a3}})$ 
we set $\zeta=\zeta_\ell$.
We assume that $t\in H(\BC)$ is regular.
If Conjecture \ref{conj:D} is valid for $\zeta$ and $t$, 
then we have an equivalence 
\[
D^b(\Mod_{\coh}(\DD_{\CB_\zeta,t}))
\cong
D^b(\Mod_{f}(U_{\zeta,t}))
\]
of triangulated  categories.
\end{proposition}
Here, for an abelian category $\CA$ we denote by $D^b(\CA)$ its bounded derived category.

\subsection{}
We have also the following conjecture for $\DD_{\CB_\zeta}$.
\begin{conjecture}
\label{conj:D0}
For $\zeta\in\BC^\times$ 
the algebra homomorphism
\eqref{eq:canon0} 
is an isomorphism and we have
\[
H^j(\CB_\zeta,\DD_{\CB_\zeta})=0
\qquad(j>0).
\]
\end{conjecture}
In \cite{T1} we proved that the localization 
$\tilde{\DD}_{\CB_\zeta}$ of $\DD_{\CB_\zeta}$ on a certain (commutative) scheme $\CV$ is an Azumaya algebra
 if $\zeta=\zeta_\ell$, where $\ell>1$ is an integer satisfying 
the conditions 
\begin{itemize}
\item[(b1)]
$\ell$ is prime to $3$ if $\Delta$ is of type $F_4$, $E_6$, $E_7$, $E_8$;
\item[(b2)]
$\ell$  is prime to
$5$ if $\Delta$ is of type $E_8$
\end{itemize}
in addition to (a1), (a2), (a3).

\begin{remark}
The condition (b2) was mistakenly dropped in \cite{T1}.
It is necessary in the proof of \cite[Lemma 6.6]{T1}.
\end{remark}
In particular, $\tilde{\DD}_{\CB_\zeta}$ is a locally free $\CO_\CV$-module if $\zeta=\zeta_\ell$ for $\ell$ satisfying (a1), (a2), (a3), (b1), (b2).
Using this we proved the following result in \cite{T2}.
\begin{proposition}
\label{prop:red}
Assume that $\zeta=\zeta_\ell$ for $\ell$ satisfying 
$({\rm{a1}})$, $({\rm{a2}})$, $({\rm{a3}})$, 
$({\rm{b1}})$, $({\rm{b2}})$.
Then Conjecture \ref{conj:D0} implies Conjecture \ref{conj:D}.
\end{proposition}

\subsection{}

Set ${}^fU_\zeta={}^fU_{\BC_\zeta}(\Gg)$.
For $t\in H$ we also set 
${}^fU_{\zeta,t}=
{}^fU_\zeta\otimes_{Z_\Har(U_\zeta)}\BC$, 
where the action of $Z_\Har(U_\zeta)$ on $\BC$ is defined by the homomorphism $\xi_{\Har,t}:Z_\Har(U_\zeta)\to\BC$.

Using ${}^fU_\zeta$ 
(resp.\ ${}^fU_{\zeta,t}$)
instead of $U_\zeta$ 
(resp.\ $U_{\zeta,t}$)
we obtain a subsheaf
${}^f\DD_{\CB_\zeta}$ 
(resp.\ ${}^f\DD_{\CB_\zeta,t}$)
of $\DD_{\CB_\zeta}$
(resp.\ ${}\DD_{\CB_\zeta,t}$).
We have natural algebra homomorphisms
\begin{align}
\label{eq:canonf0}
{}^fU_{\zeta}\otimes_{Z_{\Har}(U_\zeta)}\CO_{\BC}(H)
\to\Gamma(\CB_\zeta,{}^f\DD_{\CB_\zeta}),
\\
\label{eq:canonft}
{}^fU_{\zeta,t}\to
\Gamma(\CB_\zeta,{}^f\DD_{\CB_\zeta,t})
\qquad(t\in H).
\end{align}
We have also the following conjectures for 
${}^f\DD_{\CB_\zeta}$ and 
${}^f\DD_{\CB_\zeta,t}$.
\begin{conjecture}
\label{conj:Df0}
For $\zeta\in\BC^\times$
the algebra homomorphism
\eqref{eq:canonf0} 
is an isomorphism and we have
\[
H^j(\CB_\zeta,{}^f\DD_{\CB_\zeta})=0
\qquad(j>0).
\]
\end{conjecture}
\begin{conjecture}
\label{conj:Dft}
For $\zeta\in\BC^\times$ and $t\in H$ 
the algebra homomorphism
\eqref{eq:canonft} 
is an isomorphism and we have
\[
H^j(\CB_\zeta,{}^f\DD_{\CB_\zeta,t})=0
\qquad(j>0).
\]
\end{conjecture}
We see easily the following
\begin{lemma}
\label{lem:imp}
Conjecture \ref{conj:Df0} implies Conjecture \ref{conj:D0}, and
Conjecture \ref{conj:Dft} implies Conjecture \ref{conj:D}.
\end{lemma}
An advantage of using ${}^f\DD_{\CB_\zeta}$
and ${}^f\DD_{\CB_\zeta,t}$
 instead of 
$\DD_{\CB_\zeta}$ 
and $\DD_{\CB_\zeta,t}$
is that 
${}^f\DD_{\CB_\zeta}$ and ${}^f\DD_{\CB_\zeta,t}$ are
``$G_\zeta$-equivariant 
$\CO_{\CB_\zeta}$-modules'', where
$G_\zeta$ denotes the virtual algebraic group 
with coordinate algebra $\CO_{\BC_\zeta}(G)$.
We denote by $\Mod^{\eq}(\CO_{\CB_\zeta})$ the category of ``$G_\zeta$-equivariant 
$\CO_{\CB_\zeta}$-modules'' (see \cite[Section 4]{T2}).
Then the global section functor
\[
\Gamma(\CB_\zeta,\bullet):\Mod(\CO_{\CB_\zeta})\to\Mod(\BC)
\]
induces
\[
\Gamma(\CB_\zeta,\bullet):\Mod^\eq(\CO_{\CB_\zeta})\to\Mod_\inte(U^L_{\BC_\zeta}(\Gg)),
\]
and their higher derived functors satisfy
\[
\xymatrix@C=70pt{
\Mod^\eq(\CO_{\CB_\zeta})
\ar[r]^{H^j(\CB_\zeta,\bullet)}
\ar[d]_{\For}
&
\Mod_\inte(U^L_{\BC_\zeta}(\Gg))
\ar[d]^{\For}
\\
\Mod(\CO_{\CB_\zeta})
\ar[r]_{H^j(\CB_\zeta,\bullet)}
&
\Mod(\BC),
}
\]
where $\For$ denotes the forgetful functor.

\begin{proposition}[\cite{T2}]
\label{prop:rest1}
There exists an equivalence 
\[
\CL:\Mod^\eq(\CO_{\CB_\zeta})
\to\Mod_{\inte}(U^L_{\BC_\zeta}(\Gb^-))
\]
of abelian categories taking 
``the fiber at the origin''.
It satisfies
\[
H^j(\CB_\zeta,M)=R^j
\Ind^{\Gg,\Gb^-}_{\BC_\zeta}(\CL(M))
\qquad
(M\in\Mod^\eq(\CO_{\CB_\zeta})).
\]
\end{proposition}

\begin{proposition}[\cite{T2}]
\label{prop:rest2}
We have
\[
\CL(\DD^f_{\CB_\zeta})
=V_{\BC_\zeta},\qquad
\CL(\DD^f_{\CB_\zeta,t})
=V_{\BC_\zeta,t}.
\]
\end{proposition}

Hence from \eqref{eq:VO} and 
\eqref{eq:VOt} we obtain the following proposition.
\begin{proposition}
\label{prop:conj-eq}
Conjecture \ref{conj:Df0} 
(resp.\ Conjecture \ref{conj:Dft}) 
is equivalent to  Conjecture \ref{conj:IV} 
(resp.\ Conjecture \ref{conj:IVt}) 
in the case of $R=\BC_\zeta$.
\end{proposition}

\section{Main result}
\subsection{}
The main result of this paper is the following special case of Conjecture \ref{conj:IVt}.
\begin{theorem}
\label{thm:main}
Let $\CK$ be a commutative $\BA$-algebra which is a field, and let $t\in H(\CK)$.
Assume that there exists a Noetherian $\BA$-subalgebra $\CR$ of $\CK$ satisfying the following conditions:
\begin{itemize}
\item[(1)] $t\in H(\CR)$;
\item[(2)] the ring $\overline{\CR}=\CR/(q-1)$ is non-zero and contains a field;
\item[(3)] $\bigcap_n(q-1)^n\CR=\{0\}$.
\end{itemize}
Then the natural homomorphism
\begin{equation}
\CO_{\CK}(G)_{\ad,t}
\to 
\Ind^{\Gg,\Gb^-}_{\CK}
(\CO_\CK(B^-)_{\ad,t})
\end{equation}
is an isomorphism, and we have
\begin{equation}
R^j\Ind^{\Gg,\Gb^-}_{\CK}(\CO_\CK(B^-)_{\ad,t})
=0
\qquad(j>0).
\end{equation}
\end{theorem}
Denote by $\CO_\CK(H)_t$  the localization of 
$\CO_\CK(H)$ at the maximal ideal 
$\Ker(t:\CO_\CK(H)\to \CK)$.
\begin{theorem}
\label{thm:main1}
Let $\CK$, $t$, $\CR$ be as in Theorem \ref{thm:main}.
Then the natural homomorphism
\begin{equation}
\CO_{\CK}(G)_{\ad}
\otimes_{\CO_\CK(H)^{W\bullet}}
\CO_\CK(H)_t
\to 
\Ind^{\Gg,\Gb^-}_{\CK}(\CO_\CK(B^-)_{\ad})
\otimes_{\CO_\CK(H)}
\CO_\CK(H)_t
\end{equation}
is an isomorphism, and we have
\begin{equation}
R^j\Ind^{\Gg,\Gb^-}_{\CK}(\CO_\CK(B^-)_{\ad})
\otimes_{\CO_\CK(H)}
\CO_\CK(H)_t
=0
\qquad(j>0).
\end{equation}
\end{theorem}
Let us show that Theorem \ref{thm:main1} implies 
Theorem \ref{thm:main}.
By Lemma \ref{lem:INDfor} and Proposition \ref{prop:APW0} we have
\begin{align*}
&R^j\Ind^{\Gg,\Gb^-}_{\CK}(\CO_\CK(B^-)_{\ad})
\otimes_{\CO_\CK(H)}
\CO_\CK(H)_t
\\
\cong&R^j\Ind^{\Gg,\Gb^-}_{\CO_\CK(H)}(\CO_\CK(B^-)_{\ad})
\otimes_{\CO_\CK(H)}
\CO_\CK(H)_t
\\
\cong&
R^j\Ind^{\Gg,\Gb^-}_{\CO_\CK(H)_t}(\CO_\CK(B^-)_{\ad}
\otimes_{\CO_\CK(H)}
\CO_\CK(H)_t).
\end{align*}
Hence applying
Proposition \ref{prop:sp} to 
\[
M=\CO_\CK(B^-)_{\ad}
\otimes_{\CO_\CK(H)}
\CO_\CK(H)_t,
\quad
R=\CO_\CK(H)_t,
\quad
S=\CK
\]
we obtain Theorem \ref{thm:main} assuming 
Theorem \ref{thm:main1}.
The rest of this section is devoted the proof of Theorem \ref{thm:main1}.

\subsection{}
For 
$\lambda\in P^+$, $j\geqq0$ and a commutative $\BA$-algebra $R$ we set
\begin{align}
N^j_R(\lambda)=&
\Ext^j_{R,\Gg}(\Delta_R(\lambda),
\CO_R(G)_\ad
\otimes_{\CO_R(H)^{W\bullet}}\CO_R(H)),
\\
M_R^j(\lambda)
=&
\Ext^j_{R,\Gb^-}(\Delta_R(\lambda),
\CO_R(B^-)_\ad).
\end{align}
By Lemma \ref{lem:seG} we have 
$
\CO_R(G)_\ad
\in\Mod_{\inte}(U^L_{\CO_R(H)^{W\bullet}}(\Gg))
$, 
and hence
\[
\CO_R(G)_\ad
\otimes_{\CO_R(H)^{W\bullet}}\CO_R(H)
\in
\Mod_{\inte}(U^L_{\CO_R(H)}(\Gg)).
\]
Since $\CO_R(H)$ is a free 
$\CO_R(H)^{W\bullet}$-module (see \cite{St2}), 
we have
\begin{align*}
N^j_R(\lambda)
\cong&
\Ext^j_{\CO_R(H),\Gg}(\Delta_{\CO_R(H)}(\lambda),
\CO_R(G)_\ad
\otimes_{\CO_R(H)^{W\bullet}}\CO_R(H))
\\
\cong&
\Ext^j_{\CO_R(H)^{W\bullet},\Gg}
(\Delta_{\CO_R(H)^{W\bullet}}(\lambda),
\CO_R(G)_\ad)
\otimes_{\CO_R(H)^{W\bullet}}\CO_R(H)
\\
\cong&
\Ext^j_{R,\Gg}
(\Delta_{R}(\lambda),
\CO_R(G)_\ad)
\otimes_{\CO_R(H)^{W\bullet}}\CO_R(H)
\end{align*}
by Lemma \ref{lem:INDfor} and 
Proposition \ref{prop:APW2}.
Similarly, by Lemma \ref{lem:seB} we have 
\begin{align*}
M^j_R(\lambda)
\cong&
\Ext^j_{\CO_R(H),\Gb^-}(\Delta_{\CO_R(H)}(\lambda),
\CO_R(B^-)_\ad).
\end{align*}
Note that the canonical homomorphism
\[
\Ext^j_{R,\Gg}
(\Delta_{R}(\lambda),
\CO_R(G)_\ad)
\to
\Ext^j_{R,\Gb^-}(\Delta_{R}(\lambda),\CO_R(B^-)_\ad)
\]
is $\CO_R(H)^{W\bullet}$-linear.
Hence it induces a canonical homomorphism
\begin{equation}
\label{eq:NM}
N^j_R(\lambda)\to M^j_R(\lambda)
\end{equation}
of $\CO_R(H)$-modules.
By Proposition \ref{prop:dualWeyl2} we have the following.

\begin{lemma}
\label{lem:N-Weyl}
\begin{itemize}
\item[(i)]
Let $R$ be a commutative $\BA$-algebra.
Then for any $\lambda\in P^+$ and $j>0$  we have
$N^j_R(\lambda)=0$.
\item[(ii)]
Let $R\to S$ be a homomorphism of commutative $\BA$-algebras.
Then for any $\lambda\in P^+$ we have
$S\otimes_RN^0_R(\lambda)\cong N^0_S(\lambda)$.
\end{itemize}
\end{lemma}

\begin{lemma}
\label{lem:M-inj}
Let $\CR$ be as in Theorem \ref{thm:main}.
For any $\lambda\in P^+$, $j\geqq0$ the canonical homomorphism
\[
\overline{\CR}\otimes_\CR M^j_\CR(\lambda)\to 
M^j_{\overline{\CR}}(\lambda)
\]
is injective.
\end{lemma}
\begin{proof}
Using the exact sequence
\[
0\to \CR
\xrightarrow{q-1}
\CR
\to{\overline{\CR}}
\to 0
\]
we see that $\Tor_k^{\CR}
({\overline{\CR}},E)=0$ 
for any $\CR$-module $E$ and $k\geqq2$.
Hence our assertion is a consequence of 
Proposition \ref{prop:APW3}.
\end{proof}
\begin{proposition}
\label{prop:equiv}
Let $R$ be a commutative $\BA$-algebra, and let $T$ be a commutative flat $\CO_R(H)$-algebra.
Then the following two conditions are equivalent to each other.
\begin{itemize}
\item[(A)]
The natural homomorphism
\[
\CO_{R}(G)_{\ad}
\otimes_{\CO_R(H)^{W\bullet}}
T
\to 
\Ind^{\Gg,\Gb^-}_{R}
(\CO_R(B^-)_{\ad})
\otimes_{\CO_R(H)}
T
\]
is an isomorphism, and 
\[
R^j\Ind^{\Gg,\Gb^-}_{R}(\CO_R(B^-)_{\ad})
\otimes_{\CO_R(H)}
T
=0
\qquad(j>0).
\]
\item[(B)]
The natural homomorphism
\[
N^j_R(\lambda)\otimes_{\CO_R(H)}T
\to
M^j_R(\lambda)\otimes_{\CO_R(H)}T
\]
is an isomorphism 
for any $\lambda\in P^+$, $j\geqq0$.
\end{itemize}
\end{proposition}
\begin{proof}
Note that we have
\begin{align*}
N^j_R(\lambda)\otimes_{\CO_R(H)}T
=&
(\Ext^j_{R,\Gg}(\Delta_R(\lambda),
\CO_R(G)_\ad)\otimes_{\CO_R(H)^{W\bullet}}
{\CO_R(H)})
\otimes_{\CO_R(H)}T
\\
\cong&
\Ext^j_{R,\Gg}(\Delta_R(\lambda),
\CO_R(G)_\ad\otimes_{\CO_R(H)^{W\bullet}}T)
\\
=&
H^j(R\Hom_{R,\Gg}(
\Delta_R(\lambda),
\CO_R(G)_\ad\otimes_{\CO_R(H)^{W\bullet}}T)),
\end{align*}
and 
\begin{align*}
M^j_R(\lambda)\otimes_{\CO_R(H)}T
=&
\Ext^j_{R,\Gb^-}(
\Delta_R(\lambda),
\CO_R(B^-)_{\ad})
\otimes_{\CO_R(H)}
T
\\
\cong&
\Ext^j_{R,\Gb^-}(
\Delta_R(\lambda),
\CO_R(B^-)_{\ad}
\otimes_{\CO_R(H)}
T)
\\
=&
H^j(R\Hom_{R,\Gb^-}(
\Delta_R(\lambda),
\CO_R(B^-)_{\ad}
\otimes_{\CO_R(H)}
T))
\\
\cong&
H^j(R\Hom_{R,\Gg}(
\Delta_R(\lambda),
\Ind^{\Gg,\Gb^-}_{R}
(\CO_R(B^-)_{\ad}
\otimes_{\CO_R(H)}
T))).
\end{align*}

Assume (A).
Then we have
\begin{align*}
\CO_{R}(G)_{\ad}
\otimes_{\CO_R(H)^{W\bullet}}
T
\cong&
\Ind^{\Gg,\Gb^-}_{R}
(\CO_R(B^-)_{\ad})
\otimes_{\CO_R(H)}
T
\\
\cong&
\Ind^{\Gg,\Gb^-}_{R}
(\CO_R(B^-)_{\ad}
\otimes_{\CO_R(H)}
T),
\end{align*}
and hence (B) holds.

Assume (B).
Then we have
\begin{align}
\label{eq:NMKcon}
&R\Hom_{R,\Gg}(
\Delta_R(\lambda),
\CO_R(G)_\ad\otimes_{\CO_R(H)^{W\bullet}}T)
\\
\nonumber
\cong&
R\Hom_{R,\Gg}(
\Delta_R(\lambda),
R\Ind^{\Gg,\Gb^-}_{R}
(\CO_R(B^-)_{\ad}
\otimes_{\CO_R(H)}
T))
\end{align}
for any $\lambda\in P^+$.
Now define $L$ by the distinguished triangle
\[
\CO_R(G)_\ad\otimes_{\CO_R(H)^{W\bullet}}T
\to
R\Ind^{\Gg,\Gb^-}_{R}(\CO_R(B^-)_\ad)\otimes_{\CO_R(H)}
T
\to
L\xrightarrow{+1}.
\]
Then by \eqref{eq:NMKcon}
we have
\[
R\Hom_{R,\Gg}(\Delta_R(\lambda),L)=0
\]
for any $\lambda\in P^+$.
Assume $L\ne0$, and take the smallest $j$ satisfying
$H^j(L)\ne0$. Define $L'$ by the distinguished triangle 
\[
H^j(L)[-j]\to L\to L'\xrightarrow{+1}.
\]
Then for $k\leqq j$ we have $H^k(L')=0$, and hence
\[
\Hom_{R,\Gg}(\Delta_R(\lambda),H^j(L))=
H^j(R\Hom_{R,\Gg}(\Delta_R(\lambda),L))
=0
\qquad(\forall \lambda\in P^+).
\]
This contradicts $H^j(L)\ne0$ since $\Delta_R(\lambda)$ is the universal highest weight module with highest weight $\lambda$.
It follows that $L=0$.
Namely, we obtain
\begin{equation}
\CO_R(G)_\ad\otimes_{\CO_R(H)^{W\bullet}}T
\cong
R\Ind^{\Gg,\Gb^-}_{R}(\CO_R(B^-)_\ad)\otimes_{\CO_R(H)}T.
\end{equation}
Hence (A) holds.
\end{proof}
\subsection{}
Our strategy in proving Theorem \ref{thm:main1} is to reduce it to  the following.
\begin{proposition}
\label{prop:Fp0}
Let $R$ be a commutative ring containing a field $k$ as a subring.
We regard $R$ as an $\BA$-algebra via $q\mapsto1$.
Then the
natural homomorphism
\begin{equation}
\CO_{R}(G)_{\ad}
\otimes_{\CO_{R}(H)^{W}}
\CO_{R}(H)
\to 
\Ind^{\Gg,\Gb^-}_{R}(\CO_{R}(B^-)_{\ad})
\end{equation}
is an isomorphism, and we have
\begin{equation}
R^j\Ind^{\Gg,\Gb^-}_{R}(\CO_{R}(B^-)_{\ad})
=0
\qquad(j>0).
\end{equation}
\end{proposition}
\begin{proof}
By Proposition \ref{prop:APW0} we may assume $R=k$.
We are further reduced to the case when $k$ is an algebraically closed field by 
Proposition \ref{prop:APW0} again.
In this case the assertion is 
equivalent to a geometric statement regarding the ordinary algebraic groups as explained below.

Set $\CB_k=B_k^-\backslash G_k$ (see \ref{subsec:AlgGp} for the notation).
It is the flag manifold of $G_k$.
Denote by $\Mod^{G_k}(\CO_{\CB_k})$ the category of 
$G_k$-equivariant quasi-coherent $\CO_{\CB_k}$-modules.
This category is equivalent to the category $\Mod(B_k^-)$ of rational $B_k^-$-modules via the functor 
$\Mod^{G_k}(\CO_{\CB_k})\to\Mod(B_k^-)$ taking the fiber at $eB_k^-\in\CB_k$ (see \cite[Part I, 5.8]{Ja} for the construction of the inverse functor).
Hence by \ref{subsec:AlgGp} we have equivalences
\begin{equation}
\label{eq:ce}
\Mod_{\inte}(U^L_k(\Gg))\cong\Mod(G_k),
\qquad
\Mod_{\inte}(U^L_k(\Gb^-))\cong\Mod^{G_k}(\CO_{\CB_k})
\end{equation}
of abelian categories.
By \cite[Part I, 5.12]{Ja}
the (derived) induction functor
\[
R^j\Ind^{\Gg,\Gb^-}_{k}:
\Mod_{\inte}(U^L_k(\Gb^-))
\to
\Mod_{\inte}(U^L_k(\Gg))
\]
corresponds to
the (derived) global section functor
\[
H^j(\CB_k,\bullet):
\Mod^{G_k}(\CO_{\CB_k})
\to
\Mod(G_k)
\]
under the identification \eqref{eq:ce}.

Denote by $\CX$ the algebraic variety consisting of pairs $(B,g)$, where $B$ is a Borel subgroup of $G_{k}$ and $g$ is an element of $B$.
Let $p:\CX\to\CB_k$ be the projection $(B,g)
\mapsto B$.
Then the object of 
$\Mod^{G_k}(\CO_{\CB_k})$
corresponding to 
$\CO_k(B^-)_\ad\in\Mod_{\inte}(U^L_k(\Gb^-))$
is $p_*\CO_\CX$.
For $g\in G_{k}$ let $g_s$ be its semisimple part.
Then there exists $x\in G_{k}$ such that $xg_sx^{-1}\in H_{k}$.
Moreover, $xg_sx^{-1}$ is uniquely determined up to the action of the Weyl group $W$ on $H_{k}$.
This gives a natural morphism $G_{k}\to H_{k}/W$ of algebraic varieties.
Then the $G_k$-module corresponding to 
$\CO_{k}(G)_{\ad}
\otimes_{\CO_{k}(H)^{W}}
\CO_{k}(H)
\in
\Mod_{\inte}(U^L_k(\Gg))
$ is the coordinate algebra 
of the affine algebraic variety 
$G_k\times_{H_k/W}H_k$.
Note that 
we have a morphism 
$\theta:\CX\to G_{k}
\times_{H_{k}/W}H_{k}$ 
of algebraic varieties given by $(B,g)\mapsto (g,xg_sx^{-1})$.
Here $x\in G_{k}$ satisfies
\[
xBx^{-1}=B_{k}^-,
\qquad
xg_sx^{-1}\in H_{k}.
\]
Then our assertion is equivalent to 
\begin{equation}
\label{eq:k}
R^j\theta_*\CO_{\CX}
=
\begin{cases}
\CO_{G_{k}
\times_{H_{k}/W}H_{k}}
\quad&(j=0)
\\
0&(j>0).
\end{cases}
\end{equation}
By a general fact in algebraic geometry (see for example \cite{HO}) it is sufficient to show that 
$\theta$ is a proper birational morphism with normal image.
The fact that $\theta$ is proper and birational is standard.
Let us show that 
$Y:=G_{k}\times_{H_{k}/W}H_{k}$
is a normal variety.
Note that 
the codimension of $Y$ in $G_{k}\times H_{k}$ is $\dim H$.
On the other hand as a subvariety of $G_{k}\times H_{k}$, 
$Y$ is defined 
 by $\dim H$-equations (see \cite[3.4, Corollary 3]{St}).
Hence by Serre's criterion it is sufficient to show that $Y$ is smooth in codimension one.
It is well-known that for $g\in G_k$ we have $\dim Z_{G_k}(g)\geqq\dim H$, where $Z_{G_k}(g)$ denotes the centralizer.
Denote by $G'_k$ 
(resp.\ $G''_k$) 
the set of $g\in G_k$ satisfying 
$\dim Z_{G_k}(g)=\dim H$
(resp.\  $\dim Z_{G_k}(g)>\dim H$), 
and set 
$Y'=G'_{k}\times_{H_{k}/W}H_{k}$
(resp.\
$Y''=G''_{k}\times_{H_{k}/W}H_{k}$).
Since $Y'$ is open in $Y$,  it is sufficient to show that $Y'$ 
is smooth and $\dim Y''\leqq \dim Y-2$.
By \cite[3.8, Theorem 3]{St} $G'_k\to H_k/W$ is a smooth map.
Hence its base change $Y'\to H_k$ is also a smooth map.
Since $H_k$ is smooth, $Y'$ is smooth.
It is well-known that the morphism $Y\to H_k$ is surjective and the dimension of any fiber of it is equal to $\dim G_k-\dim H_k$.
On the other hand 
the image of $Y''\to H_k$ is a proper closed subset of $H_k$, and the dimension of any fiber of it is strictly smaller than 
$\dim G_k-\dim H_k$.
Hence we have $\dim Y''\leqq\dim Y-2$.
\end{proof}

By Proposition \ref{prop:equiv} we have the following.
\begin{proposition}
\label{prop:Fp}
Let $R$ be as in Proposition \ref{prop:Fp0}.
Then
for any $\lambda\in P^+$ and $j\geqq0$ we have
\[
N^j_{R}(\lambda)
\cong
M^j_{R}(\lambda).
\]
\end{proposition}

\subsection{}
Now we give a proof of Theorem \ref{thm:main1}.

We denote by $\overline{t}\in H(\overline{\CR})$ the image of $t\in H(\CR)$.
Choose a maximal ideal $\overline{\Gm}$ of $\CO_{\overline{\CR}}(H)$ 
containing $\Ker(\overline{t}:\CO_{\overline{\CR}}(H)\to
\overline{\CR})$, and set
$\Bk
=\CO_{\overline{\CR}}(H)/\overline{\Gm}$.
We denote the preimage of $\overline{\Gm}$ in $\CO_\CR(H)$ by $\Gm$.
Then $\Gm$ is a maximal ideal of $\CO_\CR(H)$ such that
\[
\CO_\CR(H)/\Gm\cong
\CO_{\overline{\CR}}(H)/\overline{\Gm}
\cong\Bk.
\]
We denote the localization of $\CO_\CR(H)$ 
at the maximal ideal $\Gm$ 
by
$\CO_\CR(H)_\Gm$.
Then 
$\CO_\CR(H)_\Gm$ is a local ring with 
residue field $\Bk$.
Since $\CO_\CK(H)_t$ is a localization of $\CO_\CR(H)_\Gm$, we have
\begin{align*}
N_\CK^j(\lambda)\otimes_{\CO_\CK(H)}\CO_{\CK}(H)_t
\cong&
N_\CR^j(\lambda)\otimes_{\CO_\CR(H)}\CO_\CR(H)_\Gm
\otimes_{\CO_\CR(H)_\Gm}\CO_{\CK}(H)_t,
\\
M_\CK^j(\lambda)\otimes_{\CO_\CK(H)}\CO_{\CK}(H)_t
\cong&
M_\CR^j(\lambda)\otimes_{\CO_\CR(H)}\CO_\CR(H)_\Gm
\otimes_{\CO_\CR(H)_\Gm}\CO_{\CK}(H)_t.
\end{align*}
Hence in view of Proposition \ref{prop:equiv}, 
Theorem \ref{thm:main1}  is a consequence of the following statement:
\begin{equation}
\label{eq:NMR}
N_\CR^j(\lambda)\otimes_{\CO_\CR(H)}\CO_\CR(H)_\Gm
\cong
M_\CR^j(\lambda)\otimes_{\CO_\CR(H)}\CO_\CR(H)_\Gm.
\end{equation}

Let us show \eqref{eq:NMR}.
First assume $j>0$.
Then by Lemma \ref{lem:N-Weyl} (i) it is sufficient to show 
\[
M^j_\CR(\lambda)\otimes_{\CO_\CR(H)}
\CO_\CR(H)_\Gm
=0.
\]
Note that 
$M^j_\CR(\lambda)\otimes_{\CO_\CR(H)}
\CO_\CR(H)_\Gm$ is a finitely generated 
$\CO_\CR(H)_\Gm$-module by Proposition \ref{prop:FG}.
Note also that $\CO_\CR(H)_\Gm$ is a local ring with residue field $\Bk$.
Hence by Nakayama's lemma
it is sufficient to show 
\[
(M^j_\CR(\lambda)\otimes_{\CO_\CR(H)}
\CO_\CR(H)_\Gm)
\otimes_{\CO_\CR(H)_\Gm}\Bk=0.
\]
We have
\begin{align*}
&(M^j_\CR(\lambda)\otimes_{\CO_\CR(H)}
\CO_\CR(H)_\Gm)
\otimes_{\CO_\CR(H)_\Gm}\Bk
\\
\cong&
M^j_\CR(\lambda)\otimes_{\CO_\CR(H)}
\CO_{\overline{\CR}}(H)
\otimes_{\CO_{\overline{\CR}}(H)}
{\Bk}
\\
\cong&
(M^j_\CR(\lambda)\otimes_{\CR}
{\overline{\CR}})
\otimes_{\CO_{\overline{\CR}}(H)}
{\Bk}.
\end{align*}
Hence by Lemma \ref{lem:M-inj} it is sufficient to show 
$M^j_{\overline{\CR}}(\lambda)=0$.
This follows from Lemma \ref{lem:N-Weyl} (i) and Proposition \ref{prop:Fp}.
We obtain \eqref{eq:NMR} for $j>0$.

Now consider the case $j=0$.
We first show that the canonical homomorphism
\begin{equation}
\label{eq:NM0}
N^0_\CR(\lambda)
\otimes_{\CO_\CR(H)}
\CO_\CR(H)_\Gm
\to
M^0_\CR(\lambda)
\otimes_{\CO_\CR(H)}
\CO_\CR(H)_\Gm
\end{equation}
is surjective.
Define an $\CO_\CR(H)_\Gm$-module $C$ 
to be the cokernel of \eqref{eq:NM0}.
Applying $(\bullet)\otimes_{\CO_\CR(H)_\Gm}\Bk$
to the exact sequence 
\[
N^0_\CR(\lambda)
\otimes_{\CO_\CR(H)}
\CO_\CR(H)_\Gm
\to
M^0_\CR(\lambda)
\otimes_{\CO_\CR(H)}
\CO_\CR(H)_\Gm
\to C\to 0
\]
we obtain 
\[
(N^0_\CR(\lambda)
\otimes_{\CR}\overline{\CR})
\otimes_{\CO_{\overline{\CR}}(H)}\Bk
\to
(M^0_\CR(\lambda)
\otimes_{\CR}\overline{\CR})
\otimes_{\CO_{\overline{\CR}}(H)}\Bk
\to C\otimes_{\CO_\CR(H)_\Gm}\Bk\to 0.
\]
By  Lemma \ref{lem:N-Weyl} (ii) we have
$N^0_\CR(\lambda)
\otimes_{\CR}\overline{\CR}\cong 
N^0_{\overline{\CR}}(\lambda)
$.
By Lemma \ref{lem:M-inj} the canonical homomorphism $M^0_\CR(\lambda)
\otimes_{\CR}\overline{\CR}\to
M^0_{\overline{\CR}}(\lambda)$ is injective.
Moreover, $N^0_{\overline{\CR}}(\lambda)\cong M^0_{\overline{\CR}}(\lambda)$ by Proposition \ref{prop:Fp}.
Hence by the commutative diagram 
\[
\xymatrix{
N^0_{\CR}(\lambda)\otimes_{\CR}\overline{\CR}
\ar[r]
\ar[d]_{\cong}
&
M^0_{\CR}(\lambda)\otimes_{\CR}\overline{\CR}
\ar@{^{(}-_>}[d]
\\
N^0_{\overline{\CR}}(\lambda)
\ar[r]^{\cong}
&
M^0_{\overline{\CR}}(\lambda)
}
\]
we obtain $N^0_{\CR}(\lambda)\otimes_{\CR}\overline{\CR}
\cong
M^0_{\CR}(\lambda)\otimes_{\CR}\overline{\CR}$.
Hence 
$C\otimes_{\CO_\CR(H)_\Gm}\Bk=0$.
Since $C$ is a finitely generated $\CO_\CR(H)_\Gm$-module by Proposition \ref{prop:FG}, we obtain $C=0$ by Nakayama's lemma.

It remains to show that \eqref{eq:NM0} is injective.
Since $\CO_\CR(H)_\Gm$ is a localization of $\CO_\CR(H)$, it is sufficient to show that 
$N_\CR^0(\lambda)\to M^0_\CR(\lambda)$ is injective.
In general for a commutative $\BA$-algebra $R$ we set
\[
N_{R}=\CO_{R}(G)_\ad\otimes_{\CO_{R}(H)^{W\bullet}}\CO_{R}(H),
\qquad
M_{R}=\Ind^{\Gg,\Gb^-}_{{R}}(\CO_{R}(B^-)_\ad).
\]
By
\[
N_\CR^0(\lambda)
=\Hom_{\CR,\Gg}(V_\CR(\lambda),N_\CR),
\qquad
M_\CR^0(\lambda)
=\Hom_{\CR,\Gg}(V_\CR(\lambda),M_\CR)
\]
it is sufficient to show that the canonical homomorphism $N_\CR\to M_\CR$ is injective.
Set
\begin{align*}
L=\Ker(N_\CR\to M_\CR),
\qquad
L'=\Image(N_\CR\to M_\CR).
\end{align*}
By the exact sequence
\[
0\to L\to N_\CR\to L'\to 0
\]
we obtain
\begin{equation}
\label{eq:ES}
\Tor_1^{\CR}(\overline{\CR},L')\to \overline{\CR}\otimes_{\CR}L\to 
\overline{\CR}\otimes_{\CR}N_\CR\to 
\overline{\CR}\otimes_{\CR}L'\to 0.
\end{equation}

Let us show
\begin{equation}
\label{eq:T0}
\Tor_1^{\CR}(\overline{\CR},L')=0.
\end{equation}
By the exact sequence 
\[
0\to\CR\xrightarrow{q-1}\CR\to \overline{\CR}\to0
\]
we have
\[
\Tor_1^{\CR}(\overline{\CR},L')
=\Ker(L'\xrightarrow{q-1} L').
\]
Note that $L'$ is an $\CR$-submodule of a free $\CR$-module by
\[
L'\subset \Ind^{\Gg,\Gb^-}_{\CR}(\CO_\CR(B^-)_\ad)
\subset \CO_\CR(G)\otimes_\CR\CO_\CR(B^-)_\ad.
\]
Hence we have $\Ker(L'\xrightarrow{q-1} L')=0$.
We have shown \eqref{eq:T0}.

Next let us show that 
\begin{equation}
\label{eq:INJ}
0\to \overline{\CR}\otimes_{\CR}N_\CR
\to
\overline{\CR}\otimes_{\CR}L'
\end{equation}
is exact.
Note $\overline{\CR}\otimes_{\CR}N_\CR(\lambda)\cong N_{\overline{\CR}}(\lambda)$.
By Proposition \ref{prop:Fp} we have $N_{\overline{\CR}}\cong M_{\overline{\CR}}$.
Hence \eqref{eq:INJ} is a consequence of the commutative diagram
\[
\xymatrix{
\overline{\CR}\otimes_{\CR}N_\CR
\ar[dd]_{\cong}
\ar[r]
&
\overline{\CR}\otimes_{\CR}L'
\ar[d]
\\
&
\overline{\CR}\otimes_{\CR}M_\CR
\ar[d]
\\
N_{\overline{\CR}}
\ar[r]^{\cong}
&
M_{\overline{\CR}}.
}
\]

By \eqref{eq:T0} and \eqref{eq:INJ} we see from \eqref{eq:ES} that $\overline{\CR}\otimes_\CR L=0$.
Namely, $(q-1)L=L$.
Since $L$ is an $\CR$-submodule of the free $\CR$-module  $N_\CR$, we obtain from the condition (3) that
\[
L= \bigcap_{n=0}^\infty (q-1)^nL
\subset \bigcap_{n=0}^\infty (q-1)^n N_\CR=\{0\}.
\]
Hence $N_\CR\to M_\CR$ is injective.

We have proved \eqref{eq:NMR}.
The proof of Theorem \ref{thm:main1} is now complete.

\section{Application to the representation theory}
\subsection{}
In this section we use the notation of Section \ref{sec:review}.
Let $\ell>1$ be a positive integer satisfying (a1), (a2), (a3), and set $\zeta=\zeta_\ell\in\BC^\times$.
\subsection{}
Let us first give a description of the total center $Z(U_\zeta)$ of $U_\zeta$ following \cite{DK1}, \cite{DP}.
Besides the Harish-Chandra center $Z_\Har(U_\zeta)$, which is isomorphic to $\CO_\BC(H)^{W\bullet}$,  we have a big central subalgebra $Z_{\Fr}(U_\zeta)$, called the Frobenius center.
It is naturally isomorphic to the coordinate algebra $\CO_\BC(K)$ of the algebraic group
\[
K=\{(g_+h, g_-h^{-1})\mid g_{\pm}\in N^\pm,\; h\in H\} 
\subset B^+\times B^-.
\]
Then the natural map
\[
\CO_\BC(K)\otimes\CO_\BC(H)^{W\bullet}
\cong
Z_\Fr(U_\zeta)\otimes Z_{\Har}(U_\zeta)
\to Z(U_\zeta)
\]
induces an isomorphism
\begin{equation}
Z(U_\zeta)\cong
\CO_\BC(K)\otimes_{\CO_\BC(H)^{W}}
\CO_\BC(H)^{W\bullet}
\end{equation}
of $\BC$-algebras, where the algebra homomorphisms 
\[
\CO_\BC(H)^{W}\to \CO_\BC(H)^{W\bullet},
\qquad
\CO_\BC(H)^{W}\to \CO_\BC(K)
\]
are defined as follows.
Note that $\CO_\BC(H)^{W\bullet}$ and $\CO_\BC(H)^{W}$ are naturally regarded as the coordinate algebras of the affine varieties 
$H/W\bullet$ and $H/W$ respectively.
The algebra homomorphism $\CO_\BC(H)^{W}\to \CO_\BC(H)^{W\bullet}$ corresponds to the morphism 
$
p_\ell:H/W\bullet \to H/W
$
induced by $H\ni t\mapsto t^\ell\in H$.
For $g\in G$ its semisimple part $g_{s}$ is conjugate to some $h\in H$, which is uniquely determined up to $W$-action.
This gives a morphism $\pi:G\to H/W$ of algebraic varieties.
Define a morphism $\omega:K\to G$ of algebraic varieties by 
$\omega(x_1, x_2)=x_1x_2^{-1}$.
Then the algebra homomorphism 
$\CO_\BC(H)^{W}\to \CO_\BC(K)$ 
corresponds to the morphism $\pi\circ\omega:K\to H/W$.
Hence $Z(U_\zeta)$ is isomorphic to the coordinate algebra of the affine variety
$
K\times_{H/W}{H/W\bullet}
$
with respect to $\pi\circ\omega:K\to H/W$ and $p_\ell:H/W\bullet\to H/W$.

Recall that for $t\in H$ we have 
$U_{\zeta,t}=U_\zeta/U_\zeta\Ker(\xi_{\Har,t})$, where
$\xi_{\Har,t}:Z_\Har(U_\zeta)\to\BC$ is the character corresponding to the point $[t]\in H/W\bullet$. 
For $k\in K$ define a character $\xi_{\Fr,k}:Z_{\Fr}(U_\zeta)\to\BC$ as the composit of 
$Z_{\Fr}(U_\zeta)\cong\CO_\BC(K)\xrightarrow{k}\BC$, 
and set $U_\zeta(k)=U_\zeta/U_\zeta\Ker(\xi_{\Fr,k})$.
By the above description of the center, 
$\xi_{\Har,t}$ and $\xi_{\Fr,k}$ are compatible if and only if $t^\ell$ is conjugate to $\omega(k)_{s}$ in $G$.
In this case we obtain a character $\xi_{k,t}:Z(U_\zeta)\to\BC$ corresponding to the point $(k,[t])\in K\times_{H/W}(H/W\bullet)$.
We set 
$U_\zeta(k,t)=U_\zeta/U_\zeta\Ker(\xi_{k,t})$.
\subsection{}
Now we consider the representation theory of $U_\zeta$.

A version of Schur's Lemma tells us that 
if $M$ is an irreducible $U_\zeta$-module, then any central element acts on $M$ by a scalar multiplication.
In particular, there exists some $k\in K$ such that $M$ is an irreducible $U_\zeta(k)$-module.
So we should consider the category $\Mod(U_\zeta(k))$ for each $k\in K$.
Since $U_\zeta(k)$ is finite-dimensional, any irreducible $U_\zeta$-module is finite-dimensional.

It is known from the theory of  the quantum coadjoint action (\cite{DKP}) that 
for $k, k'\in K$ we have
\begin{equation}
\label{eq:qcoad}
\omega(k)=g\omega(k')g^{-1} \;\;(\exists g\in G)
\;\;\Longrightarrow\;\;
U_\zeta(k)\cong U_\zeta(k').
\end{equation}
Hence for each conjugacy class $C$ in $G$ we only need to 
consider $\Mod(U_\zeta(k))$ for a single $k\in K$ satisfying $\omega(k)\in C$.
\begin{remark}
It is proved in \cite{DKP} that  
$U_\zeta(k)\cong U_\zeta(k')$ if there exists a conjugacy class $C$ of $G$ such that $k$ and $k'$ are contained in the same connected component of $\omega^{-1}(C)$.
This together with \cite[Theorem 16.2]{DP}
implies \eqref{eq:qcoad} when $\omega(k)$ is not contained in the center of $G$.
In the case $\omega(k)$ is contained in the center, we can directly check \eqref{eq:qcoad} using the automorphism $F$ of $U_\BF(\Gg)$ given by 
\[
F(k_\lambda)=\epsilon_\lambda k_\lambda
\quad(\lambda\in P),
\qquad
F(e_i)=e_i,
\quad
F(f_i)=\epsilon_{\alpha_i}f_i
\quad(i\in I),
\]
where $\epsilon_\lambda\in\{\pm1\}$ ($\lambda\in P$) satisfy
\[
\epsilon_\lambda\epsilon_\mu=\epsilon_{\lambda+\mu}
\quad(\lambda, \mu\in P),
\qquad
\epsilon_0=1.
\]
\end{remark}

We say that a semisimple element $g\in G$ is exceptional if the rank of the root system of its 
centralizer $Z_G(g)$ coincides with that of $\Delta$.
It follows from the theory of parabolic induction (see \cite{DK2}) that 
in considering the representation theory of $U_\zeta(k)$ we are reduced to the exceptional case.
Namely, we may assume that $\omega(k)_{s}$ is exceptional.
The representation theory in the non-exceptional  case is the same as that in the  exceptional case for a certain smaller quantized enveloping algebra.
Note that for each $G$ there exist only finitely many exceptional semisimple conjugacy classes.

We fix an exceptional semisimple element $h_0\in H$ and $k_0\in K$ such that $\omega(k_0)_{s}=h_0$.
We will  consider the representation theory of $U_\zeta(k_0,t)$ for $t\in H$ satisfying $t^\ell=h_0$ 
in the following. 
We set $G_0=Z_G(h_0)$, and denote by $\Delta_0$ the root system of $G_0$.
Then $\Delta_0$ is a root subsystem of $\Delta$ such that 
$\sum_{\alpha\in\Delta}\BQ\alpha=
\sum_{\alpha\in\Delta_0}\BQ\alpha$.
We set 
$Q_0=\sum_{\alpha\in\Delta_0}\BZ\alpha\subset Q$.
\subsection{}
Let us give an application of our results on $\DD_{\CB_\zeta,t}$ 
to the representation theory of $U_\zeta$.

In \cite{T1} we have shown under the conditions (b1), (b2) that for any $t\in H$ the sheaf 
$\DD_{\CB_\zeta,t}$ is a split Azumaya algebra over a certain central subalgebra;
however, 
in view of the arguments of \cite{T1}, in order to ensure the split Azumaya property of $\DD_{\CB_\zeta, t}$ just for $t\in H$ satisfying $t^\ell=h_0$, we have only to assume 
\begin{itemize}
\item[(a4)]
$\ell$ is prime to $Q/Q_0$
\end{itemize}
instead of (b1), (b2).
We will assume (a4) in the following.

Let us consider the Beilinson-Bernstein derived equivalence.
We obtain the following from Lemma \ref{lem:imp}, Proposition \ref{prop:conj-eq}  and Theorem \ref{thm:main}.
\begin{theorem}
\label{thm:main2}
Let $t\in H$, and assume that there exists a Noetherian $\BA$-subalgebra $\CR$ of $\BC_\zeta$ satisfying the following conditions:
\begin{itemize}
\item[(1)] $t\in H(\CR)$;
\item[(2)] the ring $\overline{\CR}=\CR/(q-1)$ is non-zero and contains a field;
\item[(3)] $\bigcap_n(q-1)^n\CR=\{0\}$.
\end{itemize}
Then
 Conjecture \ref{conj:D} holds for $\zeta=\zeta_\ell$ and $t\in H$.
 Namely,  the natural homomorphism \eqref{eq:canont} is an isomorphism and we have 
\[
H^j(\CB_\zeta,\DD_{\CB_\zeta,t})=0
\qquad(j>0). 
\]
\end{theorem}
Unfortunately an $\BA$-subalgebra $\CR$ of $\BC_\zeta$ satisfying the condition (2) of Theorem \ref{thm:main2} exists only when $\ell$ is a prime power.
In the rest of this paper we assume in addition to (a1), \dots, (a4) that
\begin{itemize}
\item[(a5)]
$\ell$ is a 
power of a prime number $p$.
\end{itemize}
\begin{lemma}
\label{lem:zen}
For $t\in H$ satisfying $t^\ell=h_0$
there exists an $\BA$-subalgebra $\CR$ of $\BC_\zeta$ satisfying the conditions $(1), (2), (3)$ of Theorem \ref{thm:main2}.
\end{lemma}
\begin{proof}
Since $h_0$ is exceptional, it has finite order $m$ which is prime to $\ell$ by (a2), (a4) (see \cite{DK2}).
Hence we have $t^{\ell m}=1$.
We set 
\[
\CR=\BZ[\zeta_{\ell m}]\subset \BC.
\]
Then $\CR$ is obviously an $\BA$-subalgebra of $\BC_\zeta$ satisfying  $t\in H(\CR)$.
By $(m,\ell)=1$ the canonical homomorphism 
$\BQ(\zeta_\ell)\otimes_\BQ\BQ(\zeta_m)\to\BQ(\zeta_{\ell m})$ induced by the multiplication is an isomorphism, and hence 
\[
\CR\cong\BZ[\zeta_\ell]\otimes_\BZ\BZ[\zeta_m]
\cong
\BZ[q^{\pm1}]/(f_\ell(q))\otimes_\BZ\BZ[\zeta_m],
\]
where $f_n(q)$ denotes the $n$-th cyclotomic polynomial.
Hence we have
\[
\overline{\CR}=\CR/(q-1)
\cong
\BZ/(f_\ell(1))\otimes_\BZ\BZ[\zeta_m]
\cong
\BF_p\otimes_\BZ\BZ[\zeta_m]
\cong
\BF_p[x]/(f_m(x)).
\]
It follows that $\overline{\CR}$ is non-zero and contains the field $\BF_p$.
Let us show
$\bigcap_n(q-1)^n\CR=\{0\}$.
By the above argument
it is sufficient to show $\bigcap_n(\zeta_\ell-1)^n\BZ[\zeta_\ell]=\{0\}$.
Assume $\bigcap_n(\zeta_\ell-1)^n\BZ[\zeta_\ell]$
contains a non-zero element $a$.
Since $\BZ[\zeta_\ell]$ is an integral domain, we can write 
uniquely $a=(\zeta_\ell-1)^na_n$ for $a_n\in\BZ[\zeta_\ell]$.
Then we have an increasing sequence 
\[
(a_1)\subset(a_2)\subset\cdots\
\]
of ideals of $\BZ[\zeta_\ell]$.
Since $\BZ[\zeta_\ell]$ is Noetherian, there exists some $n$ such that $(a_n)=(a_{n+1})$.
By $a_n=(\zeta_\ell-1)a_{n+1}$ and $a_n\ne0$ we obtain $\zeta_\ell-1\in \BZ[\zeta_\ell]^\times$, which contradicts with 
\[
\BZ[\zeta_\ell]/(\zeta_\ell-1)
\cong
(\BZ[q^{\pm1}]/(f_\ell(q)))\otimes_{\BZ[q^{\pm1}]}
\BZ[q^{\pm1}]/(q-1)
\cong
\BZ/(f_\ell(1))\cong\BF_p\ne0.
\]
\end{proof}
By Lemma \ref{lem:zen} and Proposition \ref{prop:D-eq}
we obtain the following.
\begin{theorem}
\label{thm:D-equiv}
Let $t\in H$ be a regular element satisfying $t^\ell=h_0$.
Then we have an equivalence 
\[
D^b(\Mod_{\coh}(\DD_{\CB_\zeta,t}))
\cong
D^b(\Mod_{f}(U_{\zeta,t}))
\]
of triangulated  categories.
\end{theorem}
We fix a regular element $t_0\in H$ satisfying $t_0^\ell=h_0$
in the following.
Such $t_0$ exists  if $\ell$ is sufficiently large.
For example if $h_0=1$ we can find $t_0$ as above if $\ell$ is not smaller than the Coxeter number of $G$.
The representation theory of $U_\zeta(k_0,t)$ for non-regular $t$ should be deduced from the regular case using the translation principle.

We denote by $u_0\in G_0$ the unipotent part of $\omega(k)$.
Let $B_0^-$  be a Borel subgroup of $G_0$ and set $\CB_0=B_0^-\backslash G_0$.
We define a subvariety $\CB_0^{u_0}$ of $\CB_0$ by
\[
\CB_0^{u_0}
=\{
B_0^-g\in\CB_0\mid
gu_0g^{-1}\in B_0^-\}.
\]
We note that this variety is isomorphic to the subvariety 
\[
\{B^-x\in\CB\mid
x\omega(k)x^{-1}\in h_0N^-\}
\]
of $\CB$.

Similarly to \cite{BMR}, our Theorem \ref{thm:D-equiv} together with the main result of \cite{T1} concerning the split Azumaya property implies  the following.
\begin{theorem}
\label{thm:K-gp}
We have a canonical isomorphism
\[
K(\Mod_f(U_{\zeta}(k_0,t_0)))
\cong 
K(\Mod_\coh(\CO_{\CB_0^{u_0}}))
\]
of the Grothendieck groups, where 
$\Mod_\coh(\CO_{\CB_0^{u_0}})$ denotes the category of coherent $\CO_{\CB_0^{u_0}}$-modules.
\end{theorem}

In particular, the number of simple objects in 
$\Mod_f(U_{\zeta}(k_0,t_0))$ coincides with the rank of 
$K(\Mod_\coh(\CO_{\CB_0^{u_0}}))$, 
which coincides with the dimension of the total cohomology group 
$H^*(\CB_0^{u_0},\BQ)$.

\section*{acknowledgment}
A part of this work was done while the author was staying at East China Normal University in 2019 October
as a Zijiang Professor.
I would like to thank the members of the Department of Mathematics of East China Normal University, especially Bin Shu, for their hospitality.

\bibliographystyle{unsrt}

\end{document}